\documentclass[reqno,a4paper,12pt]{amsart}
\usepackage{amsmath,amsthm,amsfonts,amssymb, mathrsfs}
\usepackage{verbatim, graphicx, ifthen, enumitem}
\usepackage[T1]{fontenc}
\usepackage [applemac] {inputenc}
\usepackage{hyperref}

\allowdisplaybreaks

\let\oldmarginpar\marginpar
\renewcommand\marginpar[1]{\-\oldmarginpar[\raggedleft\footnotesize #1]%
{\raggedright\footnotesize #1}}
\newtheorem{theorem}{Theorem}
\newtheorem{lemma}{Lemma}

\newtheorem{corollary}{Corollary}
\theoremstyle{definition}
\newtheorem{definition}{Definition}

\theoremstyle{remark}
\newtheorem{remark}{Remark}

\newcommand{\Z}{\mathbb{Z}}
\newcommand{\D}{\mathbb{D}}
\newcommand{\abs}[1]{|#1|}
\newcommand{\Abs}[1]{\left|#1\right|}
\newcommand{\Norm}[1]{\left\|#1\right\|}

\newcommand{\set}[1]{\big\{#1\big\}}

\newcommand{\N}{\mathbb{N}}
\newcommand{\R}{\mathbb{R}}

\newcommand{\C}{\mathbb{C}}

\newcommand{\Q}{\mathbb{Q}}

\def\L{\Lambda}
\def\l{\lambda}

\def\N{\mathbb{N}}
\def\Z{\mathbb{Z}}
\def\Q{\mathbb{Q}}
\def\R{\mathbb{R}}
\def\C{\mathbb{C}}
\def\F{\mathcal{F}}

\def\1{\mathbf{1}}
\def\eps{\varepsilon}

\newcommand{\dif}{\mathrm{d}}
\newcommand{\e}{\mathrm{e}}
\newcommand{\im}{\mathrm{i}}
\newcommand{\norm}[1]{\|#1\|}

\renewcommand{\Re}{\operatorname{Re}}
\renewcommand{\Im}{\operatorname{Im}}

\begin{document}

 \title[de Branges spaces with doubling phase]{Sampling and interpolation in de Branges spaces with doubling phase}

\author{Jordi Marzo, Shahaf Nitzan and Jan-Fredrik Olsen}

%\author{Jordi Marzo}
\address{Departament de matem\`atica aplicada i an\`alisi, Universitat de Barcelona, Gran Via 585, 08007, Barcelona, Spain}
\email{jmarzo@ub.edu}

%\author{Shahaf Nitzan}
%\address{Department of Mathematics, University of Missouri, Columbia, MO 65211-4100, USA}
\address{@School of Mathematical Sciences, Tel Aviv University Ramat Aviv, Tel Aviv 69978, Israel}
\email{shahaf.n.h@gmail.com}

%\author{Jan-Fredrik Olsen}
\address{Centre for Mathematical Sciences, Lund University, P.O. Box 118, SE-221 00 Lund, Sweden}
\email{janfreol@maths.lth.se}

\section*{PRELIMINARY VERSION}

\begin{abstract}	
	The de Branges spaces of entire functions generalise the classical Paley-Wiener space of 
	square summable bandlimited functions. 
	Specifically, the square norm is computed on the real line
	with respect to weights given by the values of certain entire functions.
	For the Paley-Wiener space, this can be chosen to be an  exponential function where the phase increases linearly. 
	As our main result, we   establish a natural geometric characterisation, in terms of densities, for  real sampling and interpolating sequences  in the case when the derivative of the phase function
	merely gives a doubling measure on the real line. 
	Moreover, a consequence of this doubling condition, is that the spaces we consider
	are one component model spaces.
	A novelty of our work is the application to de Branges spaces of techniques developed by Marco, Massaneda and Ortega-Cerd{\`a} for Fock spaces satisfying a doubling condition analogue to ours. 
\end{abstract}

\maketitle

%\tableofcontents

%====================================================================================
%====================================================================================
%====================================================================================
%====================================================================================
%====================================================================================

\section{Introduction}

\subsection{de Branges spaces and model spaces}

	An entire function $E$ is a Hermite-Biehler function  if it has no real zeroes and $|E^\ast(z)|<|E(z)|$ for
    $z\in \C_+=\{ z\in \C: \Im z>0 \}$, where $E^\ast(z) = \overline{E(\bar{z})}$.
    For such functions, the corresponding de Branges space is the Hilbert space of functions
\begin{equation*}
    H(E) = \left\{ f \; \text{entire} \;  : \; {f}/{E}, {f^\ast}/{E} \in H^2(\C_+) \right\},
\end{equation*}
    where $H^2(\C_+)$ is the Hardy space of the upper half-plane. The space $H(E)$ is equipped with the norm
\begin{equation*}	
    \norm{f}_{H(E)} := \sqrt{ \int_\R \Abs{ \frac{f(x)}{E(x)} }^2 \dif x}.
\end{equation*}

    The main example of a de Branges space is the Paley-Wiener space, which is
    obtained by taking $E(z)=e^{-i\pi z}$. Inspired by this, it is usual to consider the
    polar decomposition
\begin{equation*}
	E(x) = \abs{E(x)} \e^{- \im \phi(x)}.
\end{equation*}
    The phase function $\phi$ is the increasing argument of $E$.
	
    If  $Z(E)$ is the zero set of  $E$, then for some $\sigma \geq 0$ we have
\begin{equation*}
		\phi '(x)=\sigma+\sum_{z \in Z(E)}  \Im \frac{1}{x  - \overline{z}}.
\end{equation*}
	This is easy to see, e.g. since $E^\ast/E$ is an inner function on the upper half-plane with singular support only at infinity.
	Observe that when $\sigma =0$, which is an assumption we will make from now on,   the measure $\mu = \phi'(x) \dif x$ is
	a sum of harmonic measures with respect to the complex conjugate of the points in $Z(E)$.

	We say that the de Branges space $H(E)$ satisfies the doubling condition if this measure
	 is a doubling measure. I.e., if there exists a constant $C>0$ such that
	$\mu(2I) \leq C\mu(I)$ for all intervals $I \subset \R$, where $2I$ is the interval co-centric to $I$ of double length.

    The measure $\mu$ induces a natural metric on $\R$, namely
\begin{equation*}
    d_\phi(x,y) = \abs{\phi(x) - \phi(y)}.
\end{equation*}
    To motivate the significance of this metric, we consider the reproducing kernels of the space $H(E)$. These
    are functions $K_w \in H(E)$, for $w \in \C$,  such that $\langle f, K_w\rangle = f(w)$ for every $f \in H(E)$.
    It is not hard to see that in the inner product implied by the norm of $H(E)$, we have
\begin{equation*}
    K_{w}(z)=K(z,w)=\frac{i}{2\pi}\frac{E(z)\overline{E(w)}-E^{*}(z)\overline{E^{*}(w)}}{ (z-\overline{w})},\;\;z,w\in \C.
\end{equation*}
    We denote by $\widetilde{K}_{w}=K_{w}/\| K_{w} \|$ the normalised reproducing kernel.
    From this, it follows that for $x,y \in \R$ we have
\begin{equation*}
    K_y(x)  = \frac{\abs{E(x)E(y)}}{\pi}  \frac{\sin(\phi(x)-\phi(y))}{(x-y)},\;\;\widetilde{K}_y(x)=
    \frac{\abs{E(x)}\sin(\phi(x)-\phi(y))}{\sqrt{\pi\phi'(y)}(x-y)}.
\end{equation*}
	So, if we, for some fixed $\alpha \in [0,2\pi)$, let $\{ \omega_n \}$ be a sequence of points such that $\phi(\omega_n) = \alpha + \pi n$,
	then the normalised reproducing kernels $\{ \widetilde{K}_{\omega_n} \}$ form an orthonormal sequence.
	It is a theorem of de Branges   \cite[p. 55]{deBranges1968}, that for every such $\alpha$,
	except at most one, this sequence of reproducing kernels is an orthonormal basis for $H(E)$.
	%     We record that for $x \in \R$ then
% \begin{equation*}
%     \norm{K_x}^2 = \frac{\phi'(x) \abs{E(x)}^2}{\pi}.
% \end{equation*}
%     This suggests the following definition.

    The model spaces $K_\Theta^2$, where  $\Theta$ is an inner function, is the subspace of the classical Hardy space $H^2(\C_+)$  defined by 
\begin{equation*}
    K_\Theta^2  = H^2(\C_+) \cap (\Theta H^2(\C_+))^\perp.
\end{equation*}
    It is a simple observation that  $f\mapsto f/E$ is a unitary homeomorphism between $H(E)$ and
    the model subspace
    $K^2_\Theta$, where $\Theta =E^*/E$ is the inner function. In this way, it is possible to show that the de Branges spaces correspond exactly to the
      model subspaces with meromorphic inner functions \cite{havin_mashreghi2003}.
%    For these meromorphic model spaces,  we note that the reproducing kernels of $K_\Theta^2$ exist for all points $x \in \R$   and have\marginpar{'have' vs 'has'} norm $\sqrt{\phi'(x)/\pi}$.

    A model space $K_\Theta^2$ satisfies the one component condition if there exists an $\epsilon \in (0,1)$ such that
    the level set
\begin{equation*}
		\{ z : \abs{\Theta(z)} < \epsilon \}
\end{equation*}
    is connected. These spaces where introduced by Cohn \cite{cohn1982, cohn1986}.
	% It follows by a result of Alexandrov\marginpar{remove dep. on intro and abstract} (see Section \ref{??}) that if $H(E)$ satisfies the doubling condition above,
	%then the model space $K_{E^\ast/E}$ satisfies the one component condition.
    By   a characterization due to Aleksandrov \cite{aleksandrov1999} and Lemma
   \ref{local doubling equivalences} below,   the meromorphic
    model spaces satisfying the one component condition essentially correspond to  the class of de Branges spaces satisfying
    a local doubling condition (Corollary \ref{aleksandrov corollary}).

%=================================================================================================================
%=================================================================================================================
%=================================================================================================================
%=================================================================================================================

\subsection{The Fock space approach to sampling and interpolation}
    Sampling and interpolating sequences can be seen as dual concepts.
    We give the definitions for a general Hilbert space $H$ of functions, analytic
    on a domain $\Omega \subset \C$, with reproducing kernels $k_z$.
\begin{definition}
		Let $\Gamma \subset \Omega$ be a sequence of points. Then:
		\begin{itemize}
			\item[(a)] 	$\Gamma$ is called sampling for $H$ if there exists constants $A,B > 0$ such that for all $f \in H$ we have
		\begin{equation*}
			A \norm{f}^2 \leq \sum_{\gamma \in \Gamma} \abs{f(\gamma)}^2/\norm{k_\gamma}^2 \leq B\norm{f}^2.
		\end{equation*}
		(If only the upper inequality holds,   $\Gamma$ is said to be a Bessel sequence.)
		\item[(b)]  $\Gamma$ is called interpolating for $H$ if, for every sequence $\{a_\gamma\}$ satisfying
		\begin{equation*}
			\sum_{\gamma \in \Gamma} \abs{a_\gamma}^2 / \norm{k_\gamma}^2 < \infty,
		\end{equation*}
			there exist $f \in H$ such that $f(\gamma) = a_\gamma$.
		\end{itemize}
\end{definition}

    Observe that     a   sequence $\Gamma$ is sampling if and only if the
    normalized reproducing kernels $\{ k_{\gamma}/\| k_{\gamma} \| \}$ form a frame in $H$, while $\Gamma$ is
    interpolating if and only if $\{ k_{\gamma}/\| k_{\gamma} \| \}$ is a Riesz sequence in $H$ (for these, and related notions, see \cite{seip2004book}).
    
We use the following notion of separation. 
\begin{definition}
    A real sequence $\Gamma$ is called $\phi-$separated when there exists an $\epsilon>0$ such that
    $d_\phi( \gamma, \gamma') >\epsilon$	
    whenever $\gamma, \gamma' \in \Gamma$ and $\gamma\neq \gamma'$.
\end{definition}
Observe that the points $\omega_n$ for which $\phi(\omega_n) = \alpha + \pi n$ are at an integer distance from each other with respect to this metric.

        To study the \emph{real} sampling and interpolating sequences for $H(E)$, we introduce the following Beurling-type densities.
Recall that $\mu(I)=\int_I \phi'(x)dx$.
\begin{definition} 
 	Given a $\phi-$separated sequence $\Gamma=\{ \gamma_{k} \}\subset \R$, and a measure $\mu$   as above, the lower and upper uniform densities of $\Gamma$
	are given by
\begin{equation*}
	D^-_\phi(\Gamma)=\liminf_{r\to \infty} \inf_{I:\mu(I)=r} \frac{  \abs{\Gamma \cap I}}{r},
\end{equation*}
and
\begin{equation*}
	D^+_\phi(\Gamma)=\limsup_{r\to \infty} \sup_{I:\mu(I)=r} \frac{ \abs{\Gamma \cap I}}{r},
\end{equation*}
where the $I$ are bounded intervals in $\R$ and $\abs{\Gamma \cap I}$ denotes the number of points in $\Gamma \cap I$.
\end{definition}
    Our approach is motivated by   work on sampling and interpolating sequences
    in the so-called weighted Paley-Wiener spaces \cite{lyubarskii_seip2002},
    and
        in the Fock space setting \cite{marco_massaneda_ortega-cerda2003}.
    I.e., the space
\begin{equation*}
    \mathcal{F}  =  \Big\{ f \; \text{entire} : \int_\C \abs{f(z)}^2 \e^{-\Phi(z)} \dif m(z) < \infty \Big\},
\end{equation*}
    where $\dif m$ is the planar Lebesgue measure and $\Phi$ is a subharmonic function on $\C$.

    In \cite{seip1992, seip_wallsten1992}, 
    the sampling and interpolating sequences are characterized in terms of Beurling densities for the classical Fock space
    $\Phi(z)=|z|^2$ (for which $\Delta \Phi = 1$).
    In subsequent work, the $L^p$ case for spaces with $\Delta \Phi\simeq 1$ was studied  \cite{ortega-cerda_seip1998}.
    Finally, in \cite{marco_massaneda_ortega-cerda2003},
    the sampling and interpolating sequences are characterized in terms of Beurling-type densities for spaces where
    $\Delta \Phi$ gives a doubling measure on $\C$.

    The analogy to the de Branges setting becomes clearer, perhaps, when we observe that the  function
\begin{equation*}
    \Phi(z) := \left\{ \begin{array}{cc}   \log \abs{E(z)} & \Re z \geq 0,  \\ \log \abs{E^\ast(z)} & \Re z < 0,   \end{array} \right.
\end{equation*}
    has a Laplacian which is supported on $\R$ and satisfies $\Delta \Phi = 2 \phi'$ there (see Lemma \ref{delta psi lemma}). With this notation,
    the classical case (Paley-Wiener space) corresponds to $\Delta \Phi = 1$ when restricted to $\R$.
    Since functions in $H(E)$ have similar local properties as those in the corresponding Fock space, we are able to use
     Fock space techniques in our setting.

    Recall that the   density results for the Paley-Wiener case due to
    Beurling, Jaffard and Seip \cite{beurling1989, jaffard1991, seip1992} extend essentially without changes to
    the so called Weighted Paley-Wiener spaces, \cite{lyubarskii_seip2002}.
	In particular, they show that any such space is equal to a de Branges space for which the
	zeroes of the Hermite-Biehler function do not approach the real axis, and for which  $\phi'(x) \simeq 1$. The family of spaces we study contains spaces which were not studied in that setting. However, the converse is also true, they give an
	example of a de Branges space that is equal to the classical Paley-Wiener space, but for which
	the derivative of the phase of the Hermite-Biehler function is not   doubling.
%    When these spaces equal a de Branges with a Hermite-Biehler function whose zeros do not
%    approach the real axis, one can see that the
%    phase function satisfy $\phi'(x) \simeq 1$ \cite[Theorem 4]{lyubarskii_seip2002}.
%    On the other hand, when the zeros are allowed to approach the real axis, the
%    measure $\phi'(x)dx$ should be clearly not doubling.

    In an example where the phase function is made to oscillate by placing zeroes exponentially sparse on a horizontal line, 
    %and  therefore   non-doubling, 
    Baranov \cite{baranov2005stability}
  proved that, in contrast to what is known for Paley-Wiener spaces,   there are separated sequences that are not Bessel sequences. This indicates that density results   like the
    ones of theorems \ref{theorem_sampling} \and\ref{theorem_interpolation},  cannot hold in this non-doubling situation. Recently, this was studied in more generality by Belov, Mengestie and    Seip  \cite{belov_mengestie_seip2010}, who give a complete characterisation of Bessel sequences for a class of  ``small'' de Branges spaces.   (See also the discussion following the statement of Theorem \ref{theorem_interpolation}.)

%    In this case, density results like the
%    ones of theorems \ref{theorem_sampling} \and\ref{theorem_interpolation},  cannot hold. 
	
    Having said this, it appears natural to  study the de Branges spaces with doubling
    measure $\phi'(x)dx$  as an  extension of Paley-Wiener spaces,
    using techniques adapted from the Fock space setting.

    Observe that this approach suggests the study of the de Branges
      analogs to the ``large'' and ``small'' Fock spaces introduced
    in \cite{borichev_dhuez_kellay2007} and \cite{borichev_lyubarskii2010},
    among other problems. Some of this work has been already done by Mengestie in
    his forthcoming doctoral dissertation \cite{mengestie_thesis} (see also \cite{belov_mengestie_seip2010}).

%=================================================================================================================
%=================================================================================================================
%=================================================================================================================
%=================================================================================================================

\subsection{Main result}
    We obtain  the following description of sampling and interpolating sequences for the de Branges spaces
    satisfying the doubling condition.

\begin{theorem}                                             \label{theorem_sampling}
    Let $E$ be a Hermite-Biehler function with phase function $\phi$, and suppose that $\phi'(x) \dif x$ is a doubling
    measure on $\R$.

    If $\Gamma$ is a real sampling sequence for $H(E)$, then it is a finite union of $\phi-$separated subsequences
    and there exists a $\phi-$separated subset  $\Gamma'\subset \Gamma$ which is sampling with
    $D^-_\phi(\Gamma')\ge 1/\pi$. Conversely, if a real sequence $\Gamma$ is a finite union of $\phi-$separated sequences
    and there is a $\phi-$separated subsequence $\Gamma'\subset \Gamma$ such that
    $D^-_\phi(\Gamma')> 1/\pi$, then $\Gamma$ is sampling.
\end{theorem}

    The interpolation part reads as follows.

\begin{theorem}                                             \label{theorem_interpolation}
    Let $E$ be a Hermite-Biehler function with phase function $\phi$, and suppose that $\phi'(x) \dif x$ is a doubling
    measure on $\R$.

    If $\Gamma$ is a real  interpolating sequence for $H(E)$, then it is $\phi-$separated with $D^+_\phi(\Gamma)\le 1/\pi$. Conversely, if
    a real sequence
    $\Gamma$ is $\phi-$separated and $D^+_\phi(\Gamma)< 1/\pi$, then $\Gamma$ is interpolating.
\end{theorem}

    As we   pointed out above, for some spaces $H(E)$, where the derivative of the phase function $\phi$ is highly oscillating, such results are no
    longer valid \cite{baranov2005stability, belov_mengestie_seip2010}.
    % These examples are obtained by placing the zeroes
    %of the function $E$ very sparsely along  a horizontal line.
%
    An important motivation to study such de Branges spaces, was to try to find
    a counterexample to Feichtinger's conjecture, which in our setting
    says that a sampling sequence can be written as a finite union of interpolating sequences.

     However, it has been recently proved that Feichtinger's conjecture for reproducing kernels holds true   both in this case \cite{belov_mengestie_seip2010}, and in the case of  one-component
    model subspaces \cite{baranov_dyakonov2009preprint}. 
    %, and for de Branges spaces with highly oscillating derivative of the phase
    %.
    (Observe that for de Branges spaces with doubling phase, this also follows   from theorems \ref{theorem_sampling} and
    \ref{theorem_interpolation}.)
    This seems to suggest that counterexamples, if there are any, should be found when the phase goes from doubling to
    highly oscillatory.

%=================================================================================================================
%=================================================================================================================
%=================================================================================================================
%=================================================================================================================

\subsection{Outline of the paper.}
    In section \ref{sec_preliminaires}, we prove  some preliminary results
    about sampling and interpolating sequences in de Branges spaces with doubling and locally doubling measures.
    In section \ref{sec_Basic_constructions}, following Beurling we construct a multiplier function.
    In section \ref{sec_Main_results}, we first prove the necessity of the density conditions in theorems
    \ref{theorem_sampling} and  \ref{theorem_interpolation}  using the techniques developed by
    Ramanathan and Steger. Then, we prove the sufficiency for interpolation using a Lagrange interpolating function and the
    sufficiency for sampling following the weak limit techniques developed by Beurling.

%=======================================================================================================================
%=======================================================================================================================
%=======================================================================================================================

%%%%%%%%%%%%%%%%%%%%%%%%%%%%%%%%%%%%%%
%
%		Preliminaries
%
%%%%%%%%%%%%%%%%%%%%%%%%%%%%%%%%%%%%%%

\section{Preliminaries}                                                         \label{sec_preliminaires}

    We begin by establishing some preliminary results needed in the rest of the paper, some of which
    are previously known. We only provide full proofs if the result is new, or there is some novelty in the argument.

	First some notation. For $x \in \R$ and $r>0$, we set   $I_\phi (x,r)=\{ y\in \R : |\phi(x)-\phi(y)|<r \}$,
    and let $D_\phi(x,r)$ denote the disk in $\C$ with both   diameter and intersection with $\R$ equal to $I_\phi(x,r)$.
    For $\phi(x) = x$, we simply write $I(x,r)$ and $D(x,r)$, respectively. By $\dif m$, we denote the Lebesgue measure on $\C$, and we use
    the symbol $f \lesssim g$ to denote that there exists a constant $C>0$ such that $f \leq C g$ for real valued functions $f$ and $g$.
    The symbols $\gtrsim$ and $\simeq$ are defined analogously.
    For an entire function $f$,  we denote its zero set by $Z(f)$.

%%%%%%%%%%%%%%%%%%%%%%%%%%%%%%%%%%%%%%%%%%%%%%%%%%%%%%%%%%%%%%%%%%%%%%%%%%
%
%		Locally doubling measures and phase functions
%
%%%%%%%%%%%%%%%%%%%%%%%%%%%%%%%%%%%%%%%%%%%%%%%%%%%%%%%%%%%%%%%%%%%%%%%%%%

\subsection{Locally doubling measures and phase functions}

    Recall from the introduction that a positive Borel measure $\mu$ is doubling if there exists a constant $C>0$
    such that $\mu(2I) \leq C\mu(I)$ for every bounded interval $I$. We say that $\mu$ is locally doubling if
    the same holds true whenever $\mu(2I) < 1$.

    (Observe that this notion of locally doubling measures is different than the one defined with the same name in \cite{ortega-cerda2002}.)

    We state a standard result for doubling measures,
    see \cite[Lemma 2.1.]{mastroianni_totik2000} or \cite[Lemma 1]{ortega-cerda2002}.

\begin{lemma}                                               \label{lemma gamma doubling}
    Let $\mu$ be a positive (Borel) doubling measure. Then there exists $0<\gamma$ such that for any
    pair of intervals $I,I'$ with radius $r,r'$ such that $r>r',$ $I\cap I'\neq \emptyset$
\begin{equation}                                                        \label{doubling}
 	\left( \frac{\mu(I)}{\mu(I')}\right)^\gamma \lesssim \frac{r}{r'} \lesssim \left( \frac{\mu(I)}{\mu(I')}\right)^{1/\gamma}
\end{equation}
\end{lemma}

% \begin{proof}
%     Suppose that $\mu$ is doubling. Let $k\in \N$ be such that $2^{k-1}r'\le r<2^k r'.$ It is clear that $\mu(2^k I')\simeq
%      \mu(I)$. Therefore,  to prove
%     the left hand-side inequality, we restrict ourselves to the case
%     where $I=2^k I'.$ Applying iteratively the doubling condition, we get
%     $\mu(I)=\mu(2^k I')\le C^k \mu(I')$ so
%     $$\left( \frac{\mu(I)}{\mu(I')}\right)^{(\log 2)/C}\le \frac{r}{r'} .$$
%     %
%     For the right hand-side condition, 
%     we use the reversibility of the definition of doubling
%     measure. For the converse, it is enough to observe that the left hand-side inequality implies the doubling condition.
% \end{proof}

\begin{remark}
    A version     also holds for locally doubling measures.
    Indeed, one gets the same conclusion
    (\ref{doubling}) supposing that $\mu(I\cup I')\le 1$.
\end{remark}

	For meromorphic inner functions $\Theta$, it was pointed out to us by A.~Baranov that there is a close connection between the one component
	condition and the doubling condition on the phase. In the  following lemma, we essentially prove that for such functions, the local doubling condition on the phase is equivalent to the one component condition (Corollary \ref{aleksandrov corollary}).

%    The following lemma gives both a useful characterization of locally doubling measures, and
%    the connection between Hermite-Biehler functions with doubling phase and one component inner
%    functions (). 

%========================================================================================
%========================================================================================

\begin{lemma} 								\label{local doubling equivalences}
	Let $E$ be a Hermite-Biehler function with phase $\phi$. Then the following are equivalent.
	\begin{itemize}
		\item[(a)] The measure $\phi'(x) \dif x$ is locally doubling.
		\item[(b)] There exist constants such that $\phi'(x) \simeq \phi'(y)$ whenever $\abs{\phi(x) - \phi(y)} \leq 1$.
		\item[(c)] The inequality $\abs{\phi''(x)} \leq C (\phi'(x))^2$ holds.
	\end{itemize}
\end{lemma}

%========================================================================================
%========================================================================================

\begin{proof}
    $a) \iff b)$:  Suppose that $\mu$ is locally doubling, and let $C>0$ be the local doubling constant of $\mu$
    for intervals $I$ such that $\mu(2I) \leq \pi/2$. We first show that for $x,y$ in any interval $I$ of $\mu$-measure $\pi/2C$, then $\phi'(x) \simeq \phi'(y)$.

    Observe that if $J$ is an interval, co-centric with $I$, and such that $\mu(J) = \pi/2$ then by local
    doubling $\mu(J/2) \geq \mu(J)/C = \pi/2C$. This implies that $I \subset J/2$, and so $2I \subset J$. This implies $\mu(2I) \leq  \pi/2$. 

    Let $D$ denote the disk in $\C$ such that $D \cap \R = 2I$. Since $\phi'(x)$ is a sum of Poisson kernels $P_z$
    with respect to the zeros of $E^\ast$, it
    follows that no such zero   belongs to $D$. Indeed, if this was the case then the harmonic measure of $2I$ with
    respect to this point, which is smaller than $\mu(2I)$, would exceed $\pi/2$.

    For each zero $z$ of $E^\ast$, it now follows easily that $P_z(x) \simeq P_z(y)$, which yields $\phi'(x) \simeq \phi'(y)$.

    By a covering argument, we obtain $(b)$. The converse is immediate.

    $b) \iff c)$: If we denote the zeroes of $E^\ast$ by $z_n = a_n  + \im b_n$, then by the triangle inequality, 
\[
|\phi''(x)|\leq  2\sum_n\frac{b_n|x-a_n|}{\big((x-a_n)^2+b_n^2\big)^2} := 2( A + B),
\]
where we have split the sum according to whether $|x-a_n|\leq b_n$ or $|x-a_n|>b_n$, respectively.
The first estimate is clear, 
\[
A \leq \sum_{\abs{x-a_n} \leq b_n} \frac{b_n^2}{\big((x-a_n)^2+b_n^2\big)^2}\leq
\bigg(\sum_n\frac{b_n}{(x-a_n)^2+b_n^2}\bigg)^2 = (\phi'(x))^2.
\]

To obtain the second estimate, we observe that  for this case   $1\lesssim \phi'(x) |x-a_n|$.
Indeed, first suppose that $|\phi (x)-\phi (a_n)|> 1$. Letting $u$ be the point between $a_n$ and $x$ such that
$\abs{\phi(x) - \phi(u)} = 1$, it readily follows that
\begin{equation*}
	1 = \abs{\phi(x) - \phi(u)} \simeq \phi'(x) \abs{x-u} \leq \phi'(x) \abs{x-a_n}
\end{equation*}
On the other hand, if
$|\phi(x)-\phi(a_n)|\leq 1$, then we know that $\phi'(a_n)\sim\phi' (x)$. Since
$1\leq b_n \phi'(a_n)$ (always!),  the condition $b_n<|x-a_n|$
gives the required estimate.

It now follows that
\[
B \leq
\sum_{|x-a_n|>b_n}\frac{b_n}{(x-a_n)^2+b_n^2}\cdot\frac{1}{|x-a_n|}\lesssim
\phi'(x)\sum_n\frac{b_n}{(x-a_n)^2+b_n^2} = (\phi'(x))^2.\]

         The converse implication holds in general for any increasing function $\phi.$ Indeed,
         let $d_\phi(x,y)<1$ and suppose that $\phi'(x)<\phi'(y)$, then
\begin{equation*}
    \log \frac{\phi'(y)}{\phi'(x)}  =
    %  \log \frac{\phi'(y)}{\phi'(x)}=
      \left| \log  \frac{\phi'(y)}{\phi'(x)} \right|=
      \left| \int_x^y \frac{\phi''(t)}{\phi'(t)}dt   \right|
    \le
      C \left| \int_x^y \phi'(t) dt   \right|\le   C.
\end{equation*}
\end{proof}

%========================================================================================
%========================================================================================

%From Baranov:
%Condition (c) is equivalent to being one component up
%to some mild additional restricttions which are automatically satisfied when 
%\Theta is a meromorphic Blaschke product with the spectrum only at infinity.
    For inner functions $\Theta$ on the unit disk $\D$ with phase $\phi$, it is proved by Aleksandrov \cite{aleksandrov1999}
    that condition $(c)$ above, under some mild additional restrictions that are automatically satisfied when $\Theta$ is a meromorphic Blaschke product, is equivalent to $\Theta$ being one component.
%
	%In our case $\Theta = E^\ast/E$ is an inner function on the upper half-plane. 
	Following the argument of Aleksandrov, it is clear that  the same result holds on the upper half-plane as long as one makes the additional hypothesis that $\Theta  = E^\ast/E$ has infinity in its spectrum
	(cf.  \cite[Theorem 3.4]{aleksandrov1994simple}). Hence, we obtain the following consequence. 

\begin{corollary} 					\label{aleksandrov corollary}
	Let $E$ be a Hermite-Biehler function with phase $\phi$, and suppose that a subsequence of zeroes converges to infinity.
	Then $\phi'(x)\dif x$ gives a locally doubling measure if and only if the inner function $E^\ast/E$ satisfies the one component condition.
\end{corollary}

We point out that there exists a Hermite-Biehler function $E$ where the phase $\phi$ satisfies the local doubling condition, but not the doubling condition. 
This example is obtained by letting $E$ be the entire function with zeroes $z_n = - \im 2^n$ for $n \in \N$. Indeed,  then $\phi'(x) \simeq (1+ x)^{-1}$.

Rewriting parts of the two previous lemmas in the notation of the metric $d_\phi(x,y)=|\phi(x)-\phi(y)|$ we easily obtain the
following. It is analogue to \cite[Lemma 4]{marco_massaneda_ortega-cerda2003}.

%========================================================================================
%========================================================================================

\begin{lemma}                                       \label{technicallemma-2}
   Let $E$ be a Hermite-Biehler function with phase $\phi$. Then there are constants depending   on $r>0$ such that the following hold:
\begin{itemize}
\item[a)] If the measure $\phi'(x)\dif x$ is locally doubling and $d_\phi(x,y)\leq r$, then
    $d_\phi (x,y)\simeq \phi'(x)|x-y|$.
\item[b)] If the measure $\phi'(x) \dif x $ is doubling and $d_\phi (x,y)>r$, then
    $(\phi'(x)|x-y|)^{1/\alpha}\lesssim d_\phi (x,y)\lesssim
    (\phi'(x)|x-y|)^{\alpha}$, for some $\alpha>0$ depending only on the doubling constant of the measure.
\end{itemize}
\end{lemma}

%========================================================================================
%========================================================================================

\begin{remark}                                                          \label{technicallemma}
    When $\phi'(x) \dif x$ is a doubling measure, it follows by Lemma \ref{technicallemma-2}   that there exists some
    $\delta>0$, such that  for $r>0$ and  $d_\phi (x,y)>r$, then
\[\quad \phi'(y)\lesssim d_\phi (x,y)^\delta  \phi'(x),\]
    with the constant depending on $r$.
\end{remark}

\begin{remark}  \label{miamadre}
	Observe that Lemma \ref{technicallemma-2} still holds if we merely assume $\phi$ to be a smooth increasing function satisfying condition $(b)$
	of Lemma \ref{local doubling equivalences}.  
\end{remark}

%%%%%%%%%%%%%%%%%%%%%%%%%%%%%%%%%%%%%%
%
%		Pointwise estimates of functions in de Branges spaces
%
%%%%%%%%%%%%%%%%%%%%%%%%%%%%%%%%%%%%%%

\subsection{Pointwise estimates of functions in de Branges spaces}                         \label{sec_Basic_properties}

    The following lemma, while simple,  explains the analogy between de Branges
    spaces with a doubling phase function and Fock spaces with a doubling laplacian.

\begin{lemma} 											\label{delta psi lemma}
	Let $E$ be a Hermite-Biehler function, and set 
\begin{equation} 				\label{usual psi definition}
	\Phi(z) := \begin{cases} \log \abs{E(z)} & \text{if} \; \Im z \geq 0, \\ \log \abs{E^*(z)} & \text{if} \; \Im z < 0. \end{cases}
\end{equation}
	Then
	\begin{equation*}
		 \Delta \Phi(z) = \begin{cases} 2 \phi'(z) & z \in \R \\ 0 & z \notin \R. \end{cases}
	\end{equation*}
\end{lemma}

%===============================================================================================

\begin{proof}
    The function $\log E$ is analytic in a neighborhood of $x\in \R$, so by
    the Cauchy-Riemann equations, $\partial_y (\log |E|)(x)=-\partial_x (\arg E)(x)=\phi'(x)$.
    The result now follows by an argument using Green's formula.
\end{proof}

%===============================================================================================

    We now show that for functions in a de Branges space satisfying the local doubling property, a subharmonic-type estimate holds.
    This is a crucial tool, which in the context of Fock spaces corresponds to \cite[Lemma 19]{marco_massaneda_ortega-cerda2003}.

%===============================================================================================

\begin{lemma}                                                       \label{CMBsec: subharmonicity lemma}
    Let $E$ be a Hermite-Biehler function with phase function $\phi$, and let $\Phi$ be given by \eqref{usual psi definition}.
%
%
%    Let $E$ be a $HB$ function with phase $\phi$ such that
%    the measure $\phi'(x)dx$ is locally doubling. 
% 
%    Let $\Phi$ be given by (\ref{usual psi definition}).
	If $\phi'(x) \dif x$ is a locally doubling
    measure on $\R$,
    then there are constants only depending on  $r>0$ such that for all $f \in H(E)$ and $x \in \R$ we have
\begin{equation}
    \label{smooth} 	\hspace{0.7 cm} \Abs{\frac{f(x)}{E(x)}}^2    \lesssim \phi'(x)^2 \int_{D_\phi(x,r)}
    \abs{f(\xi)}^2 \e^{-2\Phi(\xi)} \dif m(\xi),
\end{equation}
	and
\begin{equation}
    \label{smoother}  \Abs{\left( \frac{f}{E}\right)'(x)}^2  \lesssim \phi'(x)^4 \int_{D_\phi(x,r)}
    \Abs{ f(\xi)}^2\e^{-2\Phi(\xi)} \dif m(\xi).
\end{equation}
\end{lemma}

%===============================================================================================

\begin{proof}

    We first observe that an inequality of the type \eqref{smooth} follows immediately
    by subharmonicity for  a Euclidean disk centered at $z \in \C$, with radius $1/\phi'(\Re z)$, and which
    does not intersect $\R$. We now prove that such an inequality holds even if the disk intersects $\R$, and deduce the lemma from it.

    Given $\rho>0$, fix $z \in \C$ and set $\Re z = x$. Suppose that the euclidean disc $D(z, \rho/\phi'(x))$ intersects $\R$.
    By Green's formula, there exists a harmonic function $h$ on $D(z, 2\rho/\phi'(x))$ such that

\begin{equation} 								\label{definition of h}
    h(\xi) = \Phi(\xi) - \Phi(z)  + \underbrace{ \int_{D(z,2\rho/\phi'(x))}
    \big( G(\eta,\xi) - G(\eta,z)  \big)  \Delta \Phi(\eta) \dif m(\eta)}_{=:B(\xi)},
\end{equation}
    where $G(\eta,\xi)$ is Green's function for $D(z,2\rho/\phi'(x))$ with pole at $\xi$. Observe that $h(z) = 0$.

    We proceed to estimate $\abs{B(\xi)}$ for $\xi \in D(z,\rho/\phi'(x))$. Set $J = D(z,2\rho/\phi'(x)) \cap \R$.  By Lemma
    \ref{delta psi lemma} we have $\Delta \Phi(\eta) =2 \phi'(\eta)$, and so

\begin{equation*}
    B(\xi)= 2 \int_{J} (G(\eta, \xi) - G(\eta,z))   \phi'(\eta) \dif \eta.
\end{equation*}
    It is clear that $J \subset I(x, 2\rho/\phi'(x))$, so by Lemma \ref{local doubling equivalences} and
    the explicit form of the Green function, we have

\begin{align*}
    \int_{J}  G(\eta,\xi)   \phi'(\eta) \dif \eta
    &\simeq
    \phi'(x)  \int_{J}  G(\eta,\xi) \dif \eta
    \leq C\rho.
\end{align*}
    Applying this estimate to both $G(\eta, z)$ and $G(\eta,\xi)$, we obtain $\abs{B(\xi)} \leq C \rho$.

    Let $H = h + \im \tilde{h}$ be an analytic completion of $h$. Since $H$ is purely imaginary at the point $z$,
    we can use the estimate on $B$, and the subharmonicity of $\abs{ \e^{-H(z)}f(z)}^2$, to obtain

\begin{equation} 								\label{smoothest}
    \abs{ {f(z)} \e^{-\Phi(z)} }^2=
    \abs{ \e^{-H(z)}f(z)}^2 \e^{-2\Phi(z)}  \lesssim \phi'(x)^2 \int_{D(z,\rho/\phi'(x))}
    \abs{f(\xi)}^2 \e^{-2\Phi(\xi)} \dif m(\xi),
\end{equation}
    where the implicit constant only depends on $\rho$.

    The inequality \eqref{smooth} now follows by observing that there exists an $\rho >0$,
    depending only on $r$, such that the ball $D(x,2\rho/\phi'(x))$ is contained in $D_\phi(x,r)$. 
	
    For \eqref{smoother}, we use a simple contour argument involving the Cauchy integral formula, and take the derivative  to see that  for any $\epsilon >0$
\begin{equation*}
	\left( \frac{f}{E} \right)'(x) = \frac{1}{2\pi \im} \int_{\abs{z-x}=\epsilon/\phi'(x)} \frac{f(z)\e^{-\Phi(z)} } {(z-x)^2} \dif z.
\end{equation*}
%\begin{align*}
%    \Abs{\left(\frac{f}{E}\right)'(x)} &= \Abs{\frac{1}{2\pi} \int_{\abs{z-x}=\epsilon/\phi'(x)} f(z)\e^{-\Phi(z)}  \frac{1}{(z-x)^2} \dif z } \\
%    &\leq  \frac{\phi'(x)^2}{2\pi\epsilon^2} \int_{\abs{z-x}=\epsilon/\phi'(x)} \Abs{ f(z)}^2\e^{-2\Phi(z)}  \dif \abs{z}.
%\end{align*}
    The conclusion now follows   by choosing  $\epsilon >0$ small enough, using the triangle inequality and applying \eqref{smoothest}.
\end{proof}

%%%%%%%%%%%%%%%%%%%%%%%%%%%%%%%%%%%%%%
%
%		Carleson measures, positive Toeplitz operators and Bernstein inequality
%
%%%%%%%%%%%%%%%%%%%%%%%%%%%%%%%%%%%%%%

\subsection{Carleson measures and a  Bernstein inequality}     \label{sec_Carleson_others}

    In this, and the next subsection, we use the Fock space point of view to obtain the
    basic tools needed to study sampling and interpolating sequences for de Branges spaces satisfying
    the doubling condition. Most of these results have been  proved by different methods in the context of one component model spaces.

    Recall that a positive Borel measure $\nu$ on $\C_+$ is called Carleson for the space $H^2(\C_+)$ if there exists a constant such that
    for all $f \in H^2(\C_+)$ 
\begin{equation*}
	\int_{\C_+} \abs{f(z)}^2 \dif \nu(z) \lesssim \norm{f}_{H^2(\C_+)}^2.
\end{equation*}
    A well-known characterization, due to Carleson \cite{carleson1962annals} and Shapiro and Shields \cite{shapiro_shields1961}, says that
    a positive Borel measure $\nu$ is a Carleson measure for $H^2(\C_+)$ if and only if $\nu(C_I) \lesssim \abs{I}$ for all bounded
    intervals $I \subset \R$,
    where $C_I$ is  the square in $\C_+$ with base   $I$. It is clear that the above holds true,
    with obvious modifications, for the space $H^2(\C_-)$.

	Carleson measures for model spaces were first considered by Cohn \cite{cohn1982, cohn1986}. For positive measures
    supported on $\R$  this notion carries over directly to de Branges spaces:  we say that a measure $\nu$
    is a Carleson measure for $H(E)$ if
\begin{equation*}
	\int_{\R} \frac{|f(x)|^2}{|E(x)|^2}d\nu(x)\lesssim \| f \|_{H(E)}^2, \qquad f \in H(E).
\end{equation*}

%=======================================================================================================

We begin with a simple lemma.

\begin{lemma}                       \label{CMBsec: carleson lemma}
    Let $E$ be a Hermite-Biehler function with phase function $\phi$, and 
    %
    %Let $\mu=\phi'(x)dx$ be a locally doubling measure in $\R$ and 
    set
 	\begin{equation*}
		D=\bigcup_{t\in \R} D_\phi(t,1).
	\end{equation*}
    If    $\phi'(x) \dif x$ is a locally doubling
    measure on $\R$,
then the measure $$\dif \nu^\pm(z) = \phi'(\Re z)\chi_{D \cap \C_\pm}(z) \dif m(z)$$ is Carleson for $H^2 (\C_\pm)$.
\end{lemma}

%=======================================================================================================

\begin{proof}
    Let $I \subset \R$ be an interval with center $t\in \R$, and denote by $C_I$ the Carleson square with base $I$. First, suppose that $\mu(I) \leq 1$.
    By Lemma \ref{local doubling equivalences}, we have $\phi'(t) \simeq \phi'(s)$ for all $s \in I$, and
    furthermore   $\abs{I} \phi'(t) \lesssim 1$. Hence,
\begin{equation*}
	\nu^+(C_I) \leq \int_{C_I} \phi'(\Re z) \dif m(z) \lesssim \phi'(t) \abs{I}^2 \lesssim \abs{I}.
\end{equation*}
	Next, suppose that $\mu(I) > 1$. Let $\{I_j\}_{j=1}^N$ be a   cover of $I$ by intervals with disjoint interiors such that $\mu(I_j)=1$. 
	It follows from the doubling condition   that
\begin{equation*}
	\abs{I} \simeq \sum_{j=1}^N \abs{I_j},
\end{equation*}
with a uniform control of the implicit constants and $N$.

	Let $\widetilde{C_{I_j}}$ denote the rectangle with base $I_j$ and height $h_j = \sup_{s \in I_j} 1/\phi'(s)$.
	Since $\phi'(t) \simeq \phi'(s)$ for $s,t \in I_j$, it now follows that
\begin{equation*}
		\nu^+ (C_I)
		\leq
		 \sum_{j=1}^N \int_{\widetilde{C_{I_j}}}\phi'(\Re z) \dif m(z)\simeq \sum_{j=1}^N \abs{I_j} \simeq \abs{I}.
\end{equation*}
    For   $\nu^-$ a similar argument works.
\end{proof}

%=======================================================================================================

\begin{remark}              \label{CMBsec: discrete carleson remark}
	A discrete version of this result also holds i.e. if $\{t_n\}$ is a $\phi-$separated sequence of points in $\R.$ Then
    $\dif \nu^{\pm}(z) = \sum_{n \in \Z} \phi'(t_n) \chi_{D_\phi(t_n, 1) \cap \C_+}(z) \dif m(z)$ 
	is a Carleson measure on $H^2(\C_{\pm})$.
\end{remark}

    The following Bernstein-type inequality is a special case of \cite[Corollary 1.5]{baranov2005bernstein}.
    We include a simple proof using our techniques.

%=======================================================================================================

\begin{lemma}                                                     \label{CMBsec: bernstein lemma}
%    Let $E$ be a Hermite-Biehler function with a phase function $\phi$.
%    Let $\Phi$ given by (\ref{usual psi definition}). If $\phi'(x)\dif x$ is
%    a locally doubling measure, 
    Let $E$ be a Hermite-Biehler function with phase function $\phi$.
 If $\phi'(x) \dif x$ is a locally doubling
    measure on $\R$,  then there exists a constant $C>0$ such that
\begin{equation*}
	\Norm{\frac{(f/E)'}{\phi'}}_{L^2 (\R)}\leq C\|f\|_{H(E)},\qquad f\in H(E).
\end{equation*}
\end{lemma}

%=======================================================================================================

\begin{proof}
    Let $D=\bigcup_{t\in \R} D_\phi(t,1)$. Observe that
    $z \in D_\phi(t,1)$ implies $t \in I_\phi(\Re z,1)$. 
    Lemma \ref{CMBsec: subharmonicity lemma}, followed by Fubini's theorem yields
\begin{align*}
    \Norm{\frac{(f/E)'}{\phi'}}^2_{L^2(\R)} &\lesssim
	\int_{\R} \phi'(t)^2 \int_{D_\phi(t,1)}|f(z)|^2e^{-2\Phi(z)}dm(z)dt \\
       &
    \lesssim
    \int_D |f(z)|^2 e^{-2\Phi(z)}\left( \int_{I (\Re z,c)} \phi'(t) dt \right) \phi'(\Re z) dm(z) \\
    &=
    \int_D |f(z)|^2 e^{-2\Phi(z)}  \phi'(\Re z) dm(z).
\end{align*}
    The conclusion now follows by Lemma \ref{CMBsec: carleson lemma}.
\end{proof}

%=======================================================================================================

\begin{remark} 
	The above lemma also holds for any $1\le p<  \infty$ using the obvious generalizations.
\end{remark}

    A well-known result by Volberg and Treil \cite{volberg_treil1988} in the context of one component model
    spaces implies the following characterization.  

%=======================================================================================================

\begin{lemma}                                                   \label{carleson_measures_lemma}
    Let $E$ be a Hermite-Biehler function with   phase function $\phi$.
    If  $\phi'(x)\dif x$ is
    a locally doubling measure on $\R$, then
    a positive Borel measure $\nu$ with support in $\R$ is a Carleson measure for the space $H(E)$ 
    if and only if
\begin{equation}                        %\label{CMBsec: carleson condition}
    \int_{I}\phi'(x)d\nu	(x)\lesssim 1
\end{equation}
    for all intervals $I \subset \R$ with $\int_{I}\phi'(x)dx=1$.
\end{lemma}

%=======================================================================================================

% \begin{proof}
% %	As the proof is similar to that of Theorem \ref{CMBsec: bernstein theorem}, we only give an outline of how to obtain it using our techniques.
%     As the result is well-known, we restrict ourselves to only indicate an easy and direct proof.
%     By the Bernstein type inequality \eqref{smoother} of Lemma \ref{CMBsec: subharmonicity lemma}, the necessity follows by a standard argument
%     (see e.g. \cite[p. 563]{ortega-cerda1998sampling}).
%     For the sufficiency, apply an argument similar to the one used to prove Lemma \ref{CMBsec: bernstein lemma}.
% \end{proof}

%%%%%%%%%%%%%%%%%%%%%%%%%%%%%%%%%%
%
%		STANDARD ARGUMENTS
%
%%%%%%%%%%%%%%%%%%%%%%%%%%%%%%%%%%

\subsection{Basic results about sampling and interpolating sequences}
    The following result is an  immediate consequence of Lemma \ref{carleson_measures_lemma}.

%=====================================================================================

\begin{corollary} \label{separation and sampling}
    Let $E$ be a Hermite-Biehler function with   phase function $\phi$.  If $\phi'(x) \dif x$ is a locally doubling measure on $\R$, then a real sequence $\Gamma$ is Bessel if and only if it is a finite union of uniformly $\phi-$separated sequences.
\end{corollary}

%=====================================================================================

    It follows that we may always assume that sampling sequences are $\phi-$separated, which is a fact we need to prove our density results.
    We now establish some other basic results.

    The next lemma is a standard tool to prove perturbative results.
    Indeed, combined with standard arguments,
    stability results for Bessel sequences, sampling sequences and interpolating sequences
    follow (see e.g. \cite[chapter 4]{young1980}).

    While we include a simple proof of the lemma to illustrate the use of our techniques, we do not
    work out all of these consequences, instead we  refer the interested reader to  Baranov \cite{baranov2005stability},
    where such results  are obtained for  general one component model spaces using other tools.

%=====================================================================================

 \begin{lemma} 											\label{SAsec: stability lemma}
   Let $E$ be a Hermite-Biehler function with   phase function $\phi$, and suppose that   $\{ \lambda_n \}$ is a real $\phi-$separated sequence.  If $\phi'(x) \dif x$ is a locally doubling measure on $\R$,
    then for every $\delta>0$ and real sequence $ \{\gamma_n\}$ that satisfies
    $\sup_{n}d_{\phi}(\l_n,\gamma_n)<\delta$, there exists a constant $C(\delta)>0$ such that
\begin{equation*}
    \sum_n |\langle f, \widetilde{K}_{\l_n}-\widetilde{K}_{\gamma_n}\rangle|^2< C(\delta) \norm{f}_{H(E)}^2,
    \qquad \forall f\in H(E).
\end{equation*}
    Moreover, $C(\delta)$ can be chosen to satisfy $C(\delta) \rightarrow 0$ as $\delta \rightarrow 0$. 
\end{lemma}

%=====================================================================================

\begin{proof}
 We begin by estimating
 \begin{align*}
	 \abs{\langle f, \widetilde{K}_{\l_n}-\widetilde{K}_{\gamma_n}\rangle}^2 &= \Abs{\frac{f(\l_n)}{E(\l_n)\sqrt{\phi'(\l_n)}}-\frac{f(\gamma_n)}{E(\gamma_n)\sqrt{\phi'(\gamma_n)}}}^2 \\
	 &=
	 \Abs{ \int_{\gamma_n}^{\lambda_n} \left( \frac{f}{E \sqrt{\phi'}}\right)'(x) \dif x }^2 \\
	 &
	 \leq  2\underbrace{\Abs{ \int_{\gamma_n}^{\lambda_n} \left( \frac{f}{E }\right)'(x) \frac{1}{\sqrt{\phi'(x)}} \dif x }^2}_{(*)}
	 +
	 \underbrace{\Abs{ \int_{\gamma_n}^{\lambda_n} \frac{f(x)}{E(x) }  \frac{\phi''(x)}{\big(\sqrt{\phi'(x)}\big)^3} \dif x }^2}_{(**)}.
\end{align*}
    To estimate $(*)$ we apply the Cauchy-Schwartz inequality to obtain
\begin{equation*}
	(*) \leq \left( \int_{\gamma_n}^{\lambda_n} \phi'(x) \dif x \right) \left( \int_{\gamma_n}^{\lambda_n} \Abs{\frac{(f/E)'(x)}{\phi'(x)}}^2\dif x \right)
	\leq \delta   \int_{\gamma_n}^{\lambda_n} \Abs{\frac{(f/E)'(x)}{\phi'(x)}}^2\dif x.
\end{equation*}
    In the case of $(**)$, the Cauchy-Schwartz inequality followed by Lemma \ref{local doubling equivalences} yields 
\begin{equation*}
	(**) \leq   \left( \int_{\gamma_n}^{\lambda_n} \phi'(x) \dif x \right) \left( \int_{\gamma_n}^{\lambda_n} \Abs{ \frac{f(x)}{E(x)} }^2 \left( \frac{\phi''(x)}{\big( \phi'(x) \big)^2}\right)^2 \dif x \right)
	\lesssim \delta   \int_{\gamma_n}^{\lambda_n} \Abs{ \frac{f(x)}{E(x)} }^2\dif x.
\end{equation*}
    Let $I_n$ denote interval with end points $\lambda_n, \gamma_n$.
For $\delta>0$ small enough, the intervals $I_n$ do not overlap, and we simply sum the previous two estimates up to get 
\begin{equation*}
	 \sum_n \abs{\langle f, \widetilde{K}_{\l_n}-\widetilde{K}_{\gamma_n}\rangle}^2
	 \lesssim   \delta \left(  \int_\R  \Abs{\frac{(f/E)'(x)}{\phi'(x)}}^2\dif x + \int_\R \Abs{\frac{f(x)}{E(x)}}^2 \dif x \right).
\end{equation*}
The desired estimate for $\delta>0$ small, now follows from Lemma \ref{CMBsec: bernstein lemma}.
To see that we get an inequality for any $\delta>0$, with some constant, it is enough to note that in general there is
a uniform bound, only depending on $\delta>0$ and $E$, for how many times the intervals $I_n$ overlap.
\end{proof}

%=====================================================================================

    We end this section on preliminaries with
    two standard results on separation that we need in the proof of theorems
    \ref{theorem_sampling} and \ref{theorem_interpolation}. 
    They do not seem to have appeared previously in our context. For a proof in the Paley-Wiener case
    see \cite[Lemma 3.11, 3.12]{Seip1995}. We provide a proof of the first lemma to illustrate again the 
    use of our techniques.

\begin{lemma}
   Let $E$ be a Hermite-Biehler function with   phase function $\phi$.  If $\phi'(x) \dif x$ is a locally doubling measure on $\R$,
	then every interpolating sequence  $\Gamma$ is $\phi-$separated.
\end{lemma}

%=====================================================================================

\begin{proof}
For $\gamma \in \Gamma$ there exists a function such that for $ {\gamma}' \in \Gamma$ we have 
\begin{equation*}
	f_\gamma(\gamma') = \begin{cases} 0 & \gamma' \neq \gamma, \\ E(\gamma) \sqrt{\phi'(\gamma)}  & \gamma' = \gamma, \end{cases}
\end{equation*}
 with $\norm{f_\gamma}_{H(E)} \lesssim 1$.
For $\gamma \neq \gamma'$, this implies 
\begin{equation*}
	 \sqrt{\phi'(\gamma)}  = \Abs{\frac{f_\gamma(\gamma)}{E(\gamma)} - \frac{f_\gamma(\gamma')}{E(\gamma')}}
	 = \abs{\gamma-\gamma'} \Abs{\left(\frac{f_\gamma}{E}\right)'(x)},
\end{equation*}
    for some $x$ in the interval with end points $\gamma, \gamma'$. By Lemma \ref{CMBsec: subharmonicity lemma}, 
    this is less than
\begin{equation*}
	 \abs{\gamma-\gamma'}  \phi'(x)^{3/2} 
    \bigg( \underbrace{ \phi'(x)\int_{D_\phi (x,1)} \Abs{ f_\gamma(\xi) }^2 \e^{-2\Phi(\xi)} \dif m (\xi)}_{(*)} \bigg)^{1/2}.
\end{equation*}
    By Remark \ref{CMBsec: discrete carleson remark}, the expression $(*)$ is bounded by some constant only depending on the norm of $f_\gamma$.  Hence,
    $1 \lesssim \phi'(x) \abs{\gamma- \gamma'}$.
    But, by Lemma \ref{technicallemma-2}, this implies that $d_\phi(\gamma,\gamma') \gtrsim 1$.
\end{proof}

%=====================================================================================

\begin{lemma}
Let $E$ be a Hermite-Biehler function with   phase function $\phi$.  If $\phi'(x) \dif x$ is a locally doubling measure on $\R$, then
for every sampling sequence $\L$, there
%If $\L$ is a sampling sequence, then there 
exists a
$\phi$-separated sequence $\L'\subset \L$ which is also a
sampling sequence.
\end{lemma}

    This result can be deduced from Lemma \ref{SAsec: stability lemma}.
    For a detailed proof in the Bergman case see \cite[p. 201]{duren_schuster2004}.

%%%%%%%%%%%%%%%%%%%%%%%%%%%%%%%%%%%%%
%
%	BASIC CONSTRUCTIONS
%
%%%%%%%%%%%%%%%%%%%%%%%%%%%%%%%%%%%%%

\section{Basic constructions: Multiplier and peak functions}                                   \label{sec_Basic_constructions}

%The next theorem, where we follow the proof of Theorem 41 in \cite{marco_massaneda_ortega-cerda2003} closely,
In this section,    we   consider a larger class of spaces than the de Branges spaces. In view of this,   let  $\Psi$ denote a real valued subharmonic function  with Laplacian supported
on the real line, and define the space
\[
\mathcal{F}_{\Psi}^{\infty}=\{f\in H(\C):
\|f\|_{\mathcal{F}_{\Psi}^{\infty}}=\sup_{z\in \C}|f(z)|e^{-\Psi
(z)}<\infty\}.
\]
Note that if $E$ is a Hermite-Biehler function and, as above, we define
\begin{equation*}
	\Phi(z) := \begin{cases} \log \abs{E(z)} & \text{if} \; \Im z > 0, \\ \log \abs{E^*(z)} & \text{if} \; \Im z < 0. \end{cases}
\end{equation*}
 then $f \in \F_{\Phi}^{\infty}$ %=H^{\infty}(E)$ in the sense that
 %a function is in $\F_{\Psi}^{\infty}$
 if and only if $f/E$ and $f^\ast/E$ are in $H^\infty(\C_+)$, with the norm computed on $\R$.

Let $\mu$ denote the measure on $\R$ given by $\Delta \Psi$. We make the additional assumptions that this measure is doubling, and that $\mu = \psi'(t) \dif t$ for some suitable function $\psi \in C^\infty(\R)$ satisfying  (b) of Lemma \ref{local doubling equivalences}, i.e., there exist constants such that the following local doubling property holds:
\begin{equation} \label{sec 3 phi prime}
	\psi'(x) \simeq \psi'(y) \; \; \text{whenever} \; \; \abs{\psi(x) - \psi(y)} \leq 1.
\end{equation}

 Recall that by $d_\psi$, $I_\psi$ and $D_\psi$ we denote the metric, open intervals and open balls induced by the measure $\psi'(t)\dif t$, respectively, and that by $Z(f)$ we denote the zero set of the entire function $f$.

%For similar
%results see \cite{}.

\subsection{Multiplier construction}

%========================================================================================

 We first establish the existence of the  so-called multiplier function. 	We follow the proof of Theorem 41 in \cite{marco_massaneda_ortega-cerda2003} closely.

\begin{theorem}                                     \label{multiplier theorem}
    Let $\Psi, \psi$ and $\mu$ be as above.	
     If  the sequence $\Lambda  = (\lambda_j)$ of real numbers is     separated in  the metric $d_\psi$ and satisfies
\begin{equation}                            \label{density condition}
		D^+_\psi(\Lambda) := \limsup_{r \rightarrow \infty} \sup_{\mu(I) = r} \frac{\abs{I \cap \Lambda}}{r} < \frac{1}{2\pi},
\end{equation}
%    where $\mu(I)=\int_I \phi'(t)dt.$
%	If $\mu$ is a doubling measure on $\R$,
	then we can construct a function
    $f \in \mathcal{F}^\infty_{\Psi}$ such that $\Lambda \subset Z(f)$, and, for some number $\eta>0$, it holds that:
\begin{itemize}
		\item[$(a)$] The set $\Lambda$ is separated from the other zeroes of $f$   in the sense that
            $D_\psi(\lambda, \eta) \cap Z(f) = \{ \lambda \}$ for all $\lambda \in \Lambda$.
 		\item[$(b)$] There exist constants such that $\abs{f(x) \e^{-\Psi(x)}} \simeq d_\psi(x,\lambda)$ %\abs{x-\lambda} \phi'(\lambda)$
             for all $\lambda \in \Lambda$
             and $x \in I_\psi(\lambda, \eta).$
\end{itemize}
\end{theorem}

	The idea in the first part of the proof, is to use \eqref{density condition} to associate
    with $\mu = \psi'(t) \dif t$ a measure on $\C$ that can be split up into pieces with good
    cancelation properties.
    To this end, we need the following lemma.

%========================================================================================

\begin{lemma}                           \label{ortega lemma}
	    Let $\tau$ be a complex measure supported in an interval
        $I \subset \R$ with $|\tau(I)|\ge 1$ and
        $\int_I |x|^j d|\tau|(x) < \infty$ for $j=1,\dots, n$.
        Then there exists a set of points $\{ \xi_1, \ldots, \xi_n \}\in \C$ such that
\begin{equation*}
	 		\frac{1}{\tau(I)}\int  x^j \dif \tau(x) = \frac{1}{n} \sum_{k=1}^n \xi_k^j,
\end{equation*}
		for $j \in \{1, \ldots, n\}$. Moreover, if $I=I(c,r)$,
        the points
        $\{ \xi_1, \ldots, \xi_n \}$
        are in a ball $D(c,C r)$ for some
        constant $C$ depending only on $n$ and $|\tau|(I).$
\end{lemma}
\begin{proof}
    We sketch a proof due to Ortega-Cerd\`a \cite{ortega-cerdaPRIVATE}. Let
\begin{equation*}
	p_j(z_1,\dots , z_n)=z_1^j+\dots + z_n^j, \qquad j=1,\dots , n.
\end{equation*}
%    be the power sum symmetric polynomials.
%
%
    Our objective is  to solve the system of equations
\begin{equation*}
	p_j(z_1,\dots, z_n)=\frac{n}{\tau(I)}\int  x^j \dif \tau(x)=c_j, \qquad j=1,\dots , n.
\end{equation*}
    We first suppose that $I=(-1,1).$ Observe that $|c_j|\le n |\tau|(I).$
    The polynomials $p_1,\dots , p_n$ are called the power sum symmetric polynomials, and it is well-known that they
    generate the ring of symmetric polynomials with rational coefficients in $n$ variables. In particular, the coefficients
    of the polynomial
    \begin{equation*}
	\prod_{i=1}^n (t  - z_i) = t^n - e_{1}(z_1, \ldots, z_n) t^{n-1} + \cdots + (-1)^n e_{n}(z_1, \ldots, z_n),
    \end{equation*}
    which are   the elementary symmetric polynomials, are in the ring of symmetric polynomials. In view of this,
%
%
%    Therefore, for any
%    elementary symmetric polynomial, $e_j(X_1,\dots , X_n),$
    there exist polynomials $Q_j(z_1, \ldots, z_n)$ such that % \in \Q[X_1,\dots , X_n]$
%    such that
	\begin{equation*}
	     e_j =Q_j(p_1,\dots , p_n), \quad j=1,\dots , n.
	\end{equation*}
    As the elementary symmetric polynomials $e_j$ also generate   the ring of symmetric polynomials in $n$ variables, there exist
%    so for
 %   $p_j(X_1,\dots , X_n),$ there exist
    polynomials $R_j(z_1,\dots, z_n)$ %\in \Z[X_1,\dots , X_n]$
    such that
	\begin{equation*}
	     p_j =R_j(e_1, \dots, e_n), \quad j=1,\dots , n.
	\end{equation*}
    Since the power sum symmetric polynomials are algebraically independent over $\Q$ \cite{macdonald1995}, it now follows   that
	\begin{equation*}
	     R_j(Q_1(z_1,\dots , z_n),\dots , Q_n(z_1,\dots , z_n))=z_j,  \quad j=1,\dots , n.
	\end{equation*}
    Let $\{ \xi_1, \ldots, \xi_n\}$ be the roots of the   polynomial
	\begin{equation*}
	     t^n-Q_1(c_1,\dots ,c_n)t^{n-1}+ \dots +(-1)^n Q_n(c_1,\dots ,c_n).
	\end{equation*}
    Clearly,
	\begin{equation*}
	    	e_j(\xi_1, \ldots, \xi_n)=Q_j(c_1,\dots ,c_n),\;\;\;j=1,\dots , n,
	\end{equation*}
    and therefore
\begin{align*}
    p_j(\xi_1, \ldots, \xi_n) & =R_j(e_1(\xi_1, \ldots, \xi_n),\dots , e_n(\xi_1, \ldots, \xi_n))
    \\
    &
    =
    R_j(Q_1(c_1,\dots ,c_n),\dots , Q_n(c_1,\dots ,c_n))=c_j.
\end{align*}
    As it holds that
	\begin{equation*}
		|\xi_j|\le \max_{j=1,\dots, n} \{ |Q_j(c_1,\dots ,c_n)|\}+1, \quad j=1,\dots, n,
	\end{equation*}
    the points $\xi_j$ are in $D(0,C)$ for some
        constant $C>0$ depending only on $n$ and $|\tau|(I)$.

    In the general case, the interval $I(c,r) = (c-r,c+r)$ can be sent to $(-1,1)$, by an affine map. If we apply  the above procedure to the measure $\tilde{\tau}$ on $(-1,1)$ induced by $\tau$, we obtain a set of points $\tilde{\xi}_k$. With this, it is not hard to see that the points
	\begin{equation*}
		\xi_k=r \tilde{\xi}_k +c,\;\;k=1,\dots, n,
	\end{equation*}
    have the desired properties.
    Indeed,
\begin{align*}
    \frac{1}{\tau(I)} & \int_I x^j d\tau(x)=\frac{1}{\tau(I)} \int_{-1}^1 (r x +c)^j d\tilde{\tau}(x)
    =
    \sum_{\ell=0}^j\binom{j}{\ell} \frac{c^{j-\ell}r^j}{\tau(I)} \int_{-1}^1 x^j d\tilde{\tau}(x)
    \\
    &
    =
    \sum_{\ell=0}^j\binom{j}{\ell} c^{j-\ell}r^j \frac{1}{n}\sum_{k=1}^n \tilde{\xi}_k^j=
    \frac{1}{n}\sum_{k=1}^n \xi_k^j.
\end{align*}
\end{proof}

%========================================================================================

\begin{proof}[Proof of Theorem \ref{multiplier theorem}]

    Given $x\in \R$, we set $\rho_x=1/\psi'(x)$. Observe that in this notation, $\mu(D (x,\rho_x))\simeq 1.$
	Define the measure
\begin{equation*}
		\widetilde{\mu} := \frac{\mu}{2\pi} - \sum_{\lambda \in \Lambda}  \delta_\lambda.
\end{equation*}
	Fix $n \in \N$ so large that $n \gamma > 1$, where $\gamma$ is
    the exponent appearing in Lemma \ref{lemma gamma doubling} for $\mu$.

	By condition (\ref{density condition}),  given any  $\epsilon >0$,    if we choose $M> 0$ large enough, then it holds
    for any interval $I \subset \R$, with $\mu(I) =2 \pi M$, that
\begin{equation*}
		\epsilon M \leq \widetilde{\mu}(I)  \leq M.
\end{equation*}
	Choose   $M$  so large that both $M \epsilon \geq n^2$ and the above holds,
	and let $\{ I_k \}$ be a partition of $\R$ into   intervals with  $\mu(I_k) =2 \pi M$.
	Note that if $M$ is an integer, then so is $\widetilde{\mu}(I_k)$.
	Hence, without loss of generality, we can add points to the set $\Lambda$, while keeping the separation
    (since no more than $M$ points have to be added to any single interval $I_k$), so that
\begin{equation*}
		\widetilde{\mu}(I_k) = n^2, \qquad \forall k \in \Z.
\end{equation*}

	Let $S_k$ be the square, symmetric with respect to $\R$,   such that $S_k \cap \R = I_k$,
    and consider the measure  $\widetilde{\mu}$ as a singularly supported measure on $\C$. Let
	$\widetilde{\mu}_k$ denote its restriction to $S_k$.  	
	
	We apply Lemma \ref{ortega lemma}  to each of the measures $\widetilde{\mu}_k$ to find
    point sets $\Xi_k = \{ \xi_{k,1}, \ldots, \xi_{k,n} \} \subset C S_k$ for which the measures
\begin{equation*}
		\widetilde{\mu}_k - n \sum_{\xi \in   \Xi_k} \delta_\xi
\end{equation*}
	each have  $n$ vanishing moments. Here $C S_k$  is the square co-centric with $S_k$ and
    with side-lengths scaled up by the factor $C.$
	
	The next step of the proof is to modify  each $\Xi_k$ to obtain sets $\Sigma_k$ such that for each $k \in \Z$, the measure
\begin{equation*}
		\nu_k =  \widetilde{\mu}_k -   \sum_{\sigma \in   \Sigma_k} \delta_\sigma
\end{equation*}
	also has $n$ vanishing moments, while keeping the set $\Lambda$    separated from $\Sigma := \cup \Sigma_k$ in the sense that there exists an $\eta>0$ such that for all
	$\lambda \in \Lambda$, the discs $D_\psi(\lambda, \eta)$ have empty intersection with  $\Sigma$. In this notation,  the set $\Sigma_k$ may have some points of multiplicity $n$.

    Because of the separation of $\Lambda,$   shrinking $\eta>0$ if necessary, we can suppose that
    the euclidean discs $D(\lambda, \eta\rho_\lambda)$ are pairwise disjoint for
    $\lambda \in \Lambda.$

	Now, for each $\xi \in   \Xi_k$ belonging to   $D(\lambda, \eta \rho_\lambda /5 )$ for some $\lambda \in \Lambda$, set
\begin{equation*}
		\Sigma_\xi = \{   \xi +  (3\eta \rho_\lambda /5 )  \e^{\frac{2\pi \im l}{n}} :  0 \leq l <n \}.
\end{equation*}
	Let $\widetilde{\Sigma}_k$ be the union of all such sets for those $\xi \in \Xi_k$. With this, define
\begin{equation*}
		\Sigma_k  =  \widetilde{\Sigma}_k \bigcup \{ \xi \in \Xi_k : \xi \notin D(\lambda, \eta \rho_\lambda), \; \forall \lambda \in \Lambda \}.
\end{equation*}
	As mentioned above, in this definition we allow some points to have multiplicity $n$.
	By construction, the set $\Sigma$ is clearly separated from $\Lambda$ in the sense
    given above (see Figure 1). To see that the measure $\nu_k$ also has $n$ vanishing moments,  it suffices
    to observe that for any $z, \tau \in \C$ and    $0 \leq j < n$ we have
\begin{equation*}
		 \sum_{l=0}^{n-1} \big(z + \tau \e^{\frac{2\pi \im l}{n}}\big) ^{j} = n z^j.
\end{equation*}
	Note that it may happen that the points in $\Sigma_k$ are now contained in some square
    larger than $C S_k$. However, since $\psi$ satisfies the local doubling property \eqref{sec 3 phi prime}, it is easily seen
    that this may be remedied by increasing $C$ slightly, independent of $k$. %the square $C S_k$ contain $\Sigma_k$.

\begin{figure}
	 	\includegraphics{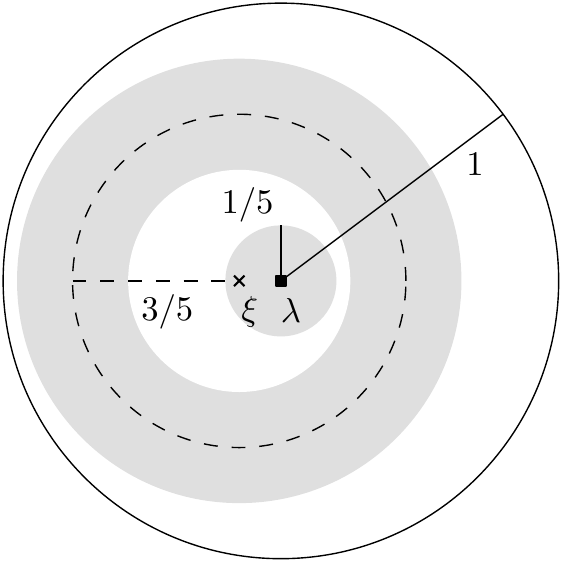}
		\caption{We illustrate how   $\Sigma_k$ is obtained by replacing  $\xi$ by   points placed on
		a circle with center $\xi$ and radius depending on the proximity to the closest $\lambda$. The
		dotted circle indicates the placement of the $n$ new points, while the shaded area indicates the new separation between
		the new points and $\lambda$.
		(Note that we have normalized the distances.)}
\end{figure}

 	To define the desired multiplier function $f$, let $\nu = \sum \nu_k$, and set
\begin{equation*}
		w(z) := \int_{\C} \log \abs{z-\zeta} \dif \nu (\zeta).
\end{equation*}
	In particular,   $\Delta w = 2\pi \nu$.
	Assume for a moment that $w(z)$ is finite for all $z \in \C$, and let $g$ denote any entire function
	with zeroes $\Sigma \cup \Lambda$ with appropriate multiplicities. Then $\Delta w = \Delta(\Psi -  \log \abs{g})$. I.e.,
\begin{equation*}
		\Delta(w - \Psi + \log\abs{g} ) = 0.
\end{equation*}
	 Hence, by Weyl's lemma, the function $w- \Psi + \log \abs{g}$ is harmonic on $\C$. Therefore,
	there exists an entire function $h$ of which this is the real part, so $w = \Re h - \log \abs{g} + \Psi$.
	We set
\begin{equation*}
		f :=  \e^{-h}g.
\end{equation*}
	
    The final step of the proof is to obtain the required estimates for $f$.	
	Indeed, to show that $f \in \mathcal{F}^\infty_\Psi$ and that $(b)$ holds, it suffices to find   positive constants $B, C, D < \infty$ for which
\begin{equation}\label{first estimate}
		  -  w(z)  \leq B, \qquad  \forall z \in \C,
\end{equation}
	and for small enough $\eta >0$ and all $\lambda \in \Lambda$, 
	% in both $\Lambda \cup \Sigma$ and $CS_{k-1} \cup CS_k \cup CS_{k+1}$, then
\begin{equation} \label{second estimate}
		- C \leq w(z) - \log \frac{R \rho_\lambda}{\abs{z- \lambda}}  \leq D, \qquad \forall z \in D(\lambda, \eta\rho_\lambda).
\end{equation}

	For the estimates below, we   let $x_{k}$ denote the center of $I_{k}$ and set $\rho_{x_k}  = \rho_k$. By the local doubling property \eqref{sec 3 phi prime}, it follows readily that there is a radius $R>0$
	such that for all $k \in \Z$ we have $CS_{k-1} \cup CS_k \cup CS_{k+1} \subset D(x_k, R \rho_k )$.

	We first prove that \eqref{first estimate} holds.
	Fix $z \notin \cup_k D(x_k, 2R\rho_k)$ and let $\zeta\in D(x_k, R\rho_k).$
	In this case, $\abs{\zeta-x_k} \leq R \rho_k  \leq \abs{z-x_k}$.
    By $P_k (\zeta)=P_{k,n-1,z} (\zeta)$, denote  the Taylor polynomial of order $n-1$ of the function $\log (z-\zeta)$ at the point $x_k$.  Then,  since $|z-\zeta|\ge |z-x_k|/2$, it holds that
	\begin{equation*}
		\left| \log(z-\zeta)-P_k (\zeta)\right|\le \max_{\xi\in <x_k,\zeta>} \left|  \frac{(\zeta-x_k)^n }{(z-\xi)^n n} \right|
    \lesssim  \frac{\rho_k^n}{|z-x_k|^n },
	\end{equation*}
    where $<x_k,\zeta>$ is the segment in $\C$ with endpoints $x_k$ and $\zeta$.
	So,
    by the $n$ vanishing moments of $\nu_k$,
\begin{align*}
		\Abs{\int_\C  \log \abs{z-\zeta} \dif \nu(\zeta)} &=\Abs{\Re \int_\C  \log (z-\zeta) \dif \nu(\zeta)}
        \le \Abs{ \int_\C  \log (z-\zeta) \dif \nu(\zeta)}
        \\
        \le
        &
        \Abs{ \sum_{k \in \Z} \int_{CS_k}  \log (z-\zeta) \dif \nu_k(\zeta) }\le
         \sum_{k \in \Z} \Abs{  \int_{CS_k}  (\log(z-\zeta)-P_k (\zeta)) \dif \nu_k(\zeta) }
        \\
		&
        \lesssim
        \sum_{k \in \Z}  \frac{\rho_k^n }{\abs{z-x_k}^n }   \int_{CS_k}    \dif \abs{ \nu_k}(\zeta)
		\lesssim
        \sum_{k \in \Z}  \frac{\rho_k^n }{\abs{z-x_k}^n }.
\end{align*}
	Let   $k_z \in \N$ be such that $\Re z \in I_{k_z}$.
	To see that the sum converges  with a bound independent of $z$,   it suffices to establish that
    $\abs{z-x_k} \geq A\rho_k(\abs{k - k_z}^\gamma + 1)$,
    for some constant $A$  only depending on $R$.
	Recall that   $n \gamma > 1$, and $\gamma$ is
    the exponent appearing in Lemma \ref{lemma gamma doubling} for $\mu$.
    We leave it to the reader to find that such an estimate follows from the doubling property of $\mu$.

	Next, suppose that $z \in D(x_{k_z}, 2R\rho_{{k_z}})$ for some $k_z \in \N$.
    By the doubling property, there exists a number $q$, such that for $\abs{k-k_z} \geq q$,
    then $z \notin D(x_k,2R \rho_k )$. So, if we make the split
\begin{equation*}
		w(z) = \underbrace{\sum_{\abs{k-k_z}\leq q} \int_\C \log \abs{z-\zeta } \dif
    \nu_k(\zeta)}_{:=w_1(z)}+  \underbrace{\sum_{\abs{k-k_z}> q} \int_\C \log \abs{z-\zeta} \dif \nu_k(\zeta)}_{:=w_2(z)},
\end{equation*}
	then we can use the same method to get a bound $\abs{w_2(z)} \le B$  as we did in the first case above.
	To estimate   $w_1(z)$, we increase $R$ if necessary so that 
	  $\cup_{\abs{k-k_z}\leq q} CS_k \subset D(x_{k_z},R \rho_{k_z})$ (this is possible independently of $k$).
    Then we can use the first vanishing moment of $\nu_k$ to get
\begin{align*}
		- w_1(z) & =
        \sum_{\abs{k-k_z}\leq q} \int_{CS_k} \log \frac{R \rho_{k_z}}{\abs{z-\zeta}}
        \dif\left( \frac{\mu_k}{2\pi}  -  \sum_{\lambda \in \Lambda \cap I_k} \delta_\lambda -  \sum_{\sigma \in \Sigma_k}
        \delta_\sigma \right)(\zeta)
        \\
		&
        \lesssim
		\int_{D(x_{k_z}, R \rho_{k_z})} \log \frac{R \rho_{k_z} }{\abs{z-\zeta}} \dif \mu (\zeta).
\end{align*}
	Since $\dif \mu = \psi'(t) \dif t$, and $\psi'(t) \simeq \psi'(x_{k_z})$ for $t \in I(x_{k_z}, R \rho_{k_z})$,
    it follows by basic calculus that the
	integral on the right-hand side is comparable to $R$.

    We turn to the demonstration of \eqref{second estimate}.  Consider $\lambda \in \Lambda$ and $z \in D(\lambda, \eta \rho_\lambda)$, where
      $\eta>0$ is a  separation constant of $\Lambda$ from $\Sigma$.
 To get the lower estimate, we make 
    the same split   $w(z)  = w_1(z) + w_2(z)$ as above, using $\lambda$ in place of $k_z$. The
	estimate of $w_2(z)$ follows as before with the obvious modifications. In the lower estimate of $w_1(z)$, recall that above the expression
\begin{equation*}
    I := \sum_{\abs{k-k_z}\leq q} \int_{CS_k} \log \frac{R \rho_\lambda}{\abs{z-\zeta}}
    \dif\left(   \sum_{\lambda \in \Lambda \cap S_k} \delta_\lambda +   \sum_{\sigma \in \Sigma_k} \delta_\sigma \right)(\zeta)
\end{equation*}
	was estimated   simply as $I \geq 0$. The lower estimate of $w_1(z)$ follows in the current case, if   we replace this by
\begin{equation*} 
	I =  \sum_{\abs{k-k_z}\leq q} \sum_{\gamma \in (\Lambda \cap I_k) \cup \Sigma_k}
    \log \frac{R \rho_\lambda}{\abs{z-\gamma}}  \geq \log \frac{R \rho_{\lambda}}{ \abs{z-\lambda}}.
    % \simeq \log \frac{R \rho_{\lambda}}{ \abs{z-\lambda}} .
\end{equation*}
	As for the upper estimate of $w_1(z)$, since $\mu$ is a positive measure,
\begin{equation*}
		w_1(z)
		\leq  \log \frac{R \rho_{\lambda}}{\abs{z-\lambda}}  +
        \sum_{\abs{k-k_z} \leq q} \sum_{\gamma \in \Sigma_k \cup \Lambda\cap I_k \backslash \{ \lambda \}} \log \frac{R \rho_\lambda}{\abs{z-\gamma}}.
\end{equation*}
	It follows easily that the whole sum is bounded by some uniform constant, since every $\gamma$   on the right-hand side is at least at a euclidean distance of $\eta  \rho_\lambda/2$  from $z$.  This ends the proof. 
\end{proof}

%=======================================================================

\subsection{Peak functions}                                 \label{subsec_peak_functions}

As an application of the above multiplier construction, we obtain
a class of peak functions analogue to the ones used by Beurling in
his original work on sampling sequences.
First,
we record the following simple lemma. % \marginpar{There is a way to simplify by making it more obscure...}

%=======================================================================
%\marginpar{hmm... screw the hermite-biehler guy.. also, is this lemma correct? $\phi$ vs $a \phi$. Smarter way to do this... just replace $0$ by $x_0$... compactness argument even...}
\begin{lemma}                                               \label{multiplier corollary}
%Let $E$ be an Hermite-Biehler
%function with phase $\phi$ which gives a doubling measure, and let
%$\Psi$ be as in (\ref{usual psi definition}).
    Let $\Psi$ and  $\psi$  be as above.
Then, for every $x_0\in \R$ there exist a function $f \in
\F_{\Psi}^{\infty}$ and a sequence $\L\subset Z(f)\cap\R$
that satisfy the following conditions, all with constants not  
depending on $x_0$:
\begin{itemize}
\item[(i)] The set $\L\cup\{x_0\}$ is separated.
\item[(ii)]  The quantity $|f(x_0)|e^{-\Psi (x_0)}$ is bounded below.
\item[(iii)] The map $x \longmapsto d_\psi(x,\L)$ is bounded above.
\item[(iv)] We have % the following estimate holds
%\[
$|f(x)|\lesssim d_{\psi}(x,\L)e^{\Psi (x)}$
%\]
for every $x \in \R$.
 \end{itemize}
\end{lemma}

%=======================================================================

\begin{proof}
    Set $\Lambda = \set{\lambda \in \R : d_{\psi}(x_0, \lambda) = 4\pi k,  0\neq k \in \Z}$.
    Since $D^+_{\psi}(\Lambda) =1/4\pi$ we may use
    these $\L$ to
    repeat the construction from Theorem 3. We do this while keeping the
    sequence $\Sigma$ separated not only from $\L$ but also from
    $\L\cup\{x_0\}$. This can be done in the same way as in the
    proof of Theorem 3, with a constant of separation not
    depending on $x_0$. Note that the constants $B, C$ and $D$
    appearing in (10) and (11) depend only on the distribution of
    $\L$ (with respect to $d_{\psi}$) and the doubling constant of
    $\mu=\psi'(x)dx$. Hence, the bounds obtained in these
    constructions do not depend on $x_0$. Properties (i), (iii)
    and (iv) follow immediately. To see that
    property (ii) is satisfied we make an estimate similar to the one on the right
    side of (11), with $z=x_0$, and use the fact that $x_0$ is
    uniformly
    separated from all points in $\L\cup\Sigma$.
\end{proof}

%=======================================================================

We are now ready to construct our peak functions.
\begin{lemma} \label{peak lemma}
    Let $\Psi$ and $\psi$ be as above, and fix $a>0$ and $k \in \N$.
%    Let $E$ be a Hermite-Biehler function with phase function $\phi$, and suppose that $\phi'(t) \dif t$ is a doubling
%    measure on $\R$. Moreover, let $\Phi$ be as above and fix $\epsilon>0$ and $k \in \N$.
%
%
%Fix $\Psi$ as above and let $\eps>0$ and $k\in\N$.
For every $x \in \R$ there is a function $h(x,\cdot)\in
\F_{a\Psi}^{\infty}$ that satisfies $h(x, x)=1$ and
\begin{equation} \label{cond on h}
    \abs{h(x, y)} \leq C \frac{1}{1+d^k_{\psi}(x,y)}e^{a(\Psi(y)-\Psi(x))}\qquad \forall y\in\R,
\end{equation}
where $C$ is a constant not depending on $x$. Moreover, the
following estimate holds whenever $k$ is big
enough:
\begin{equation}  \label{int estimate for h}
 \sup_{y\in
\R}\psi'(y)\int_{\R}|h(x, y)|^2e^{2a(\Psi(x)-\Psi(y))}dx< \infty.
\end{equation}
\end{lemma}

%=======================================================================

\begin{proof}
For $x_0\in \R$, let $f \in
\F_{a\Psi}^{\infty}$ be the function from Lemma \ref{multiplier
corollary}. Property (iii) of the lemma
implies that there exists $L>0$, not depending on $x_0$, such that
the interval $I_{\psi}(x_0,L)$ includes at least $M$ points from
the corresponding sequence $\L$. Choose such
$\sigma_1,...,\sigma_M\in \L$ and set
\[
h(x_0,z)=\frac{f(z)}{(z-\sigma_1)\cdot\cdot\cdot(z-\sigma_M)}\frac{c(x_0)e^{-\eps\Psi(x_0)}}{(\phi'(x_0))^M}
\]
where $c(x_0)$ is chosen to satisfy $h(x_0,x_0)=1$. It is clear
that $h(x_0,z) \in \F_{a\Psi}^{\infty}$ since outside of a big
enough disc we have $|h(x_0,z)|\lesssim |f(z)|$. The inequality
\eqref{cond on h} is satisfied whenever $M$ is big enough,
  as in the proof of \cite[Theorem 18]{marco_massaneda_ortega-cerda2003}. To prove \eqref{int estimate for h},
we use \eqref{cond on h} to see that  for  $x, y\in \R$  we have
\begin{align*}
\psi'(y) \int_{\R}|h(x,y)|^2e^{2a(\Psi(x)-\Psi(y))}&dx \lesssim
\int_{\R}\frac{\psi'(y)dx}{1+d^{2k}_{\psi}(x,y)} \\
&=  \left[ \int_{I_\psi(y,1)}+\int_{\R/I_\psi(y,1)}
\right]\frac{\psi'(y)dx}{1+d^{2k}_{\psi}(x,y)},
\end{align*}
where the constant does not depend on $y$. For the first integral
on the right-hand side, the needed estimate  is easily obtained.
%\marginpar{It is no longer this lemma... maybe put something in a remark after the lemma... yup.. that this is also true for any absolutely continuous doubling guy...}
Next, since $\psi'(x)\dif x$ is a doubling measure, by Remark \ref{technicallemma}, %as part $(b)$ of Lemma \ref{technicallemma-2} holds for
there exists a $\delta >0$ such that the second integral is
dominated by
\[
%\int_{\R/I(y,1)}\frac{\psi'(y)dx}{1+d^{2k}_{\psi}(x,y)}  \lesssim
\int_{\R}\frac{d^{\delta}_{\psi}(x,y)\psi'(x)dx}{1+d^{2k}_{\psi}(x,y)}\lesssim
\int_{\R}\frac{\psi'(x)dx}{1+d^{2k-\delta}_{\psi}(x,y)}.
\]
By the change of variables $t=\psi(x)-\psi(y)$, this is seen to
converge for large $k$.
\end{proof}

\section{Proofs of main results}          \label{sec_Main_results}

\subsection{Necessary conditions for sampling and interpolation}

	To get the necessary conditions for a real sequence of points to be interpolating or sampling, we use the
    technique developed by Ramanathan and Steger \cite{ramanathan_steger1995} (see also \cite[Lemma 4]{ortega-cerda_seip1998}). Since we have canonical orthonormal bases in de Branges spaces, this theorem is easily applicable. 

%=====================================================================================

\begin{theorem}
    Let $E$ be a Hermite-Biehler function with phase function $\phi$, and suppose that $\phi'(x) \dif x$ is a doubling
    measure on $\R$.

	If $\Gamma=\{ \gamma_{k} \}\subset \R$ is a separated sampling sequence for $H(E)$
	and   $\Lambda=\{ \lambda_{k} \}\subset \R $ is an interpolating sequence for $H(E)$, then
$D^+_\phi(\Lambda)\le D^+_\phi(\Gamma)$ and $D_\phi^-(\Lambda) \le D_\phi^-(\Gamma)$.
\end{theorem}

%=====================================================================================

\begin{proof}
	Recall that   $K_z$ is the reproducing kernel of $H(E)$ at the point $z \in \C$, while $\widetilde{K}_z = K_z/\norm{K_z}_{H(E)}$.
    We denote by $\{ G_{k} \}$ and $\{ L_{k}\}$ the dual frame of $\{ \widetilde{K}_{\gamma_{k}} \}$, and the biorthogonal basis of $\{ \widetilde{K}_{\lambda_{k}} \}$, respectively.

    For $t,r>0,$ $x\in \R$ we define the sets 
\begin{equation*}
	W_\Gamma (x)= \mathrm{span} \, \{ G_{k} : \gamma_k\in I_\phi(x,t+r) \},\;\;\; W_\Lambda (x)=  \mathrm{span} \, \{ \widetilde{K}_{\lambda_k} : \lambda_k\in I_\phi(x,t)\},
\end{equation*}
    and the corresponding orthogonal projections $P_\Gamma$ and $P_\Lambda$ onto $W_\Gamma(x)$ and $W_\Lambda(x)$, respectively.

	The idea of the proof is the following observation. It is clearly sufficient   to show that,   given any $\epsilon>0$, then for $r >0$ big enough, and all $t, x \in \R$, the trace of the operator
	$T=P_\Lambda P_\Gamma {P_\Lambda}$, as an operator on $W_\Lambda(x)$,
    satisfies
\begin{equation}                                                        \label{homo}
	 		(1-\epsilon) \abs{ \Lambda \cap I_\phi(x,t)} \leq tr(T)
			\leq
			 \abs{ \Gamma\cap I_\phi(x,t+r)}.
\end{equation}
	
     On the finite dimensional space $W_\Lambda (x)$, the  trace of  the self-adjoint and positive operator $T$ is computed by
$$tr (T)=\sum_{\lambda_k \in I_\phi (x,t)} \langle T \widetilde{K}_{\lambda_k},L_k \rangle.$$
	The upper estimate follows easily from the fact that eigenvalues of $T$ are bounded by one in modulus:
$$tr (T)\le \mbox{rank} (T)\le \dim W_\Gamma (x)\le \abs{ \Gamma\cap I_\phi(x,t+r)}.$$
	For the lower estimate,
	we use the Cauchy-Schwartz inequality and the biorthogonality property of $\{ L_k \}$ to get
	$$\langle T \widetilde{K}_{\lambda_k},L_k \rangle \ge 1- \| P_\Gamma \widetilde{K}_{\lambda_k}-\widetilde{K}_{\lambda_k} \|_{H(E)} \| P_\Lambda L_k \|_{H(E)}.$$
	The term $\norm{P_\Lambda L_k}_{H(E)}$ is uniformly bounded, therefore it only remains to make  $\| P_\Gamma \widetilde{K}_{\lambda_k}-\widetilde{K}_{\lambda_k} \|_{H(E)}$
	small.
	Since $\Gamma$ is sampling, a simple computation involving the projection $P_\Gamma$ yields
\begin{equation*}
	    \| P_\Gamma \widetilde{K}_{\lambda_k}-\widetilde{K}_{\lambda_k} \|_{H(E)}^2
	    \lesssim
	       \sum_{\gamma_n\not \in I_\phi(x,t+r)} |\langle \widetilde{K}_{\lambda_k},\widetilde{K}_{\gamma_n}\rangle|^{2}.
 \end{equation*}
    Now let $\lambda_k \in I_\phi(x,t),$ $\gamma_n\not \in I_\phi(x,t+r)$ and let $\delta>0$ be the separation constant of $\Gamma$ with respect to the metric $d_\phi$.
	Inserting the expression of our reproducing kernels, and making  a simple estimate, this is      bounded above by
\begin{equation*}
	\frac{1}{\phi'(\lambda_k)}
    \sum_{\gamma_n\not \in I_\phi(x,t+r)} \int_{I_\phi (\gamma_n,\delta) } \frac{ds}{(\lambda_k-s)^2}
    \lesssim
    \frac{1}{\phi'(\lambda_k) }\int_{ \R \setminus I_\phi (\lambda_k,r/2)}\frac{ds}{(\lambda_k-s)^2}\lesssim \frac{1}{r^{1/\alpha}}
\end{equation*}
	The left-most inequality holds   as soon as $r > 2\delta$, while the right-most holds
	for  the  $\alpha >1$ given by Lemma \ref{technicallemma-2}.
	With this, \eqref{homo} follows.
\end{proof}

%=========================================================================
%=========================================================================
%=========================================================================
%=========================================================================
%=========================================================================
%=========================================================================

\subsection{Sufficiency for interpolation}
	For a de Branges space $H(E)$ with $\Phi$ defined as in \eqref{usual psi definition}, we recall that $\Delta \Phi = 2 \phi'(x) \dif x$. So with $\Psi = \Phi$,
    given a $\phi$-separated sequence $\Lambda \subset \R$ for which $D^+_\phi(\Lambda)<1/\pi,$
    it follows from Theorem \ref{multiplier theorem} that for small enough $\eta >0$ there exists a function
    $f \in \mathcal{F}^\infty_{\Phi}$ such that uniformly in  $\lambda \in \Lambda$ we have
\begin{equation}                                                                \label{estimate on the derivative}
    \abs{f'(\lambda)} \simeq \abs{E(\lambda)} \phi'(\lambda).
\end{equation}
    and
\begin{align}
		 \abs{(f/E)(x)/(x-\lambda)} \simeq \phi'(\lambda),\quad x \in I_\phi(\lambda, \eta)      \label{uniform estimate on the derivative}
%		 \abs{(f/E)(z)/(z-\lambda)} \simeq \phi'(\lambda),\;\;z \in D_\phi(\lambda, \eta)\cap \C_+        \label{uniform estimate on the derivative} \\
% 		 \abs{(f/E^*)(z)/(z-\lambda)} \simeq \phi'(\lambda),\;\;z \in D_\phi(\lambda, \eta)\cap \C_-.      \label{uniform estimate on the derivative_inf}
\end{align}
 
We now use  this multiplier $f$ to solve the interpolation problem $F(\lambda) = w_\lambda$, where the data $(w_\lambda)$  satisfies
$\sum_{\lambda \in \Lambda} \abs{w_\lambda}^2 / {\abs{E(\lambda)}^{2} \phi'(\lambda)} < \infty$.
Our objective is to show that the solution is given by  the following Lagrange-type interpolation function:
\begin{equation*}
	F(z) = \sum_{\lambda \in \Lambda} w_\lambda \frac{f(z)}{(z-\lambda) f'(\lambda)}
\end{equation*}
We proceed by duality. First, 
\begin{equation*}
	\norm{F}_{H(E)}
	= \sup_{\underset{\norm{h}_2 = 1}{h \in H^2(\C_+)}} \Abs{\sum  \frac{w_\lambda}{f'(\lambda)} \int_\R \overline{h(t)} \frac{f(t)/E(t)}{t-\lambda} \dif t}.
\end{equation*}
At each term of this sum, we multiply and divide by $\sqrt{\phi'(\lambda)} \abs{E(\lambda)}$.

The Cauchy-Schwarz inequality, in combination with the assumption on $\{w_\lambda\}$ and  \eqref{estimate on the derivative},   yields
\begin{equation} \label{valid inequality}
	\norm{F}_{H(E)}^2 \lesssim \sup_{\underset{\norm{h}_{2}= 1}{h \in H^2(\C_+)}} 
	 \sum_\lambda \frac{1}{\phi'(\lambda)} \Abs{H\left(  \frac{\bar{h} f}{E}  \right)(\lambda)}^2,
\end{equation}
where $H$ denotes the Hilbert transform   on the real line.

Given $\lambda \in \Lambda$, we consider the decomposition
\begin{equation} \label{decomposed hilbert}
	 H\left(  \frac{\bar{h} f}{E}  \right)(\lambda)= \left( \int_{d_\phi(t,\lambda)<\eta }
    +  \int_{d_\phi(t,\lambda)>\eta}  \right) \overline{h(t)} \frac{f(t)/E(t)}{t-\lambda} \dif t  := A + B.
\end{equation}
With this, the Cauchy-Schwarz inequality, Lemma \ref{technicallemma-2} and  \eqref{uniform estimate on the derivative}, imply
\begin{equation} \label{local estimate}
	\abs{A}^2
	 \lesssim \phi'(\lambda) \int_{d_\phi(t,\lambda)<\eta} \abs{h(t)}^2 \dif t.
\end{equation}
As for $B$,  we disconnect the $\lambda$ appearing in the domain of integration and in the function, and define
\begin{equation*}
	 B (\xi) := \int_{d_\phi(t,\lambda)>\eta} \overline{h(t)} \frac{f(t)/E(t)}{t-\xi} \dif t.
\end{equation*}
This function is analytic for $\xi \in D_\phi(\lambda, \eta)$.
In a similar way we define the function $A(\xi)$, which is, however,  only analytic for $\xi \notin \R$.
By subharmonicity,  and since the radius of $D_\phi(\lambda,\eta/2)$ is comparable to $1/\phi'(\lambda)$, we get the inequality
\begin{align*}
	 \abs{B}^2 &\lesssim  {\phi'(\lambda)^2}  \int_{D_\phi(\lambda, \eta/2)} \Abs{ B(\xi)  }^2 \dif m(\xi) \\
	 &\leq {\phi'(\lambda)^2}  \int_{D_\phi(\lambda, \eta/2)} \Abs{H\left(  \frac{\bar{h} f}{E}  \right)(\xi)}^2  \dif m(\xi)
	   +
	    {\phi'(\lambda)^2}  \underbrace{   \int_{D_\phi(\lambda, \eta/2)} \Abs{
	    A(\xi)}^2
	    \dif    m(\xi) }_{:= C}.
\end{align*}
Since $f/E$ is bounded on $\R$, and the Hilbert transform is bounded on $L^2(\R)$, we use the fact that $\phi'(\lambda) \chi_{D_\phi(\lambda, \eta/2)}$ gives a Carleson measure
    for both $H^2(\C_+)$ and $H^2(\C_-)$, with constant only depending on $\eta$, to see that
   \begin{equation*}
	\phi'(\lambda) {C}
	  \lesssim
	  \Norm{ \frac{\bar{h}f }{E} \chi_{I_\phi(\lambda,\eta)}}^2_{L^2(\R)}
	\lesssim  \int_{I_\phi(\lambda,\eta)} \abs{h(t)}^2 \dif t.
\end{equation*}
Plugging these estimates into  \eqref{decomposed hilbert}, yields
\begin{equation*}
	\Abs{	 H\left(  \frac{\bar{h} f}{E}  \right)(\lambda)}^2 \lesssim
    \phi'(\lambda)^2 \int_{D_\phi(\lambda, \eta/2)} \Abs{ H\left(  \frac{\bar{h} f}{E}  \right)(\xi) }^2  \dif m(\xi)
	+ \phi'(\lambda) \int_{d_\phi(t,\lambda) < \eta} \abs{h(t)}^2 \dif t.
\end{equation*}
Due to the $\phi$-separation of the points $\lambda$ we get
\begin{equation*}
	\sum_{\lambda \in \Lambda} \int_{d_\phi(t,\lambda)< \eta} \abs{h(t)}^2 \dif t \leq \norm{h}^2_{L^2(\R)}.
\end{equation*}
 In addition, by Lemma \ref{CMBsec: carleson lemma}, 
\begin{equation*}
	\sum_{\lambda \in \Lambda} \phi'(\lambda)  \int_{D_\phi(\lambda,\eta/2)} \Abs{ H\left(  \frac{\bar{h} f}{E}  \right)(\xi)  }^2  \dif m(\xi)
	\lesssim
	\Norm{H\left(  \frac{\bar{h} f}{E}  \right)}^2_{H^2(\C_+)} + \Norm{H\left(  \frac{\bar{h} f}{E}  \right)}^2_{H^2(\C_-)}.
\end{equation*}
This yields the desired conclusion  as the Cauchy transform is bounded on $L^2(\R)$, and $f/E$ is bounded on $\R$. % this is bounded by $\norm{h}^2_2$.

%=========================================================================
%=========================================================================
%=========================================================================
%=========================================================================
%=========================================================================
%=========================================================================

\subsection{Sufficiency for sampling}

We follow the ideas of Beurling in the case of the Paley-Wiener space \cite{beurling1989}, as adapted to  the Fock space setting by Marco, Massaneda and Ortega-Cerd{\`a} in \cite{marco_massaneda_ortega-cerda2003}. 

For the space $\mathcal{F}_{\Psi}^{\infty}$, defined in Section \ref{sec_Basic_constructions},
%let  $\Psi$ denote a real valued subharmonic function  with Laplacian supported
%on the real line.
we say that a sequence $\L \subseteq \R$ is a sampling sequence
%for $\mathcal{F}_{\Psi}^{\infty}$
if there exists some constant $C>0$
such that for all $f\in
\mathcal{F}_{\Psi}^{\infty}$ we have
\[
\|f\|_{\mathcal{F}_{\Psi}^{\infty}}\leq C\sup_{\l\in
\L}|f(\l)|e^{-\Psi (\l)}.
\]

The first lemma states that in order to prove that a sequence
is sampling in a de Branges space, it is enough to prove that it
is sampling in such a uniform norm space. 
%
%We remark that the proof differs from the one of Beurling in the Paley-Wiener case \cite{beurling1989}, since functions
%from $H(E)$ are in general not bounded (i.e., $f/E$ and $f^*/E$
%in general do not belong to $H^\infty (\C^+)$).
\begin{lemma}                                                   \label{tango}
    Let $E$ be a Hermite-Biehler function with phase function $\phi$,   suppose that $\phi'(x) \dif x$ is a doubling
    measure on $\R$ and let $\Phi$ be defined by \eqref{usual psi definition}. Moreover, suppose that $\L$ is a
$\phi$-separated sequence.

%Let $E$ be an Hermite-Biehler function with phase $\phi$ which gives a doubling
%measure and let $\Psi$ be as in (\ref{usual psi definition}).
%Assume that $\L$ is a
%$\phi$-separated sequence.
 If, for some $\eps>0$, the sequence $\L$
is a sampling sequence for $\F_{(1+\epsilon)\Phi}^{\infty}$, then
it is a sampling sequence for $H(E)$.
\end{lemma}
\begin{proof}
It is sufficient to show that if $\L$ is a sampling
sequence for $\F_{(1+\epsilon)\Phi}^{\infty}$, then every $g\in
H(E)$ is pointwise equal to a series of the form $\sum
a_{\l}\widetilde{K}_{\l}$ with $\sum|a_{\l}|^2\lesssim
\|g\|^2_{H(E)}$, where $K_w$ is the reproducing kernel in $H(E)$ at $w \in \R$.

So, assume that for some $\epsilon>0$, the sequence $\L$ is a
sampling sequence for $\F_{(1+\epsilon)\Phi}^{\infty}$.  Following  a duality
argument in  \cite[pp. 348-358]{beurling1989},
 there exist functions
$g_{\lambda}(x)$, with $\sum_{\l\in
\L}|g_{\lambda}(x)|\leq K$ uniformly in $x$, so that for all
%\[
$f \in \mathcal{F}_{(1+\epsilon)\Phi}^{\infty}$ such that $|f(x)|e^{-(1+\epsilon)\Phi (x)}\rightarrow 0$   as  $|x|\rightarrow
\infty$,
%\]
%$f\in \mathcal{F}_{\Phi}^{\infty,0}$
 we have 
\begin{equation}\label{the condition for func in F^0}
f(x)e^{-(1+\epsilon)\Phi (x)}=\sum_{\l \in \L}f(\l)e^{-(1+\epsilon)\Phi
(\l)}g_{\l}(x)\qquad \forall x \in \R.
\end{equation}

Let $h(x,\cdot)$ be the function from Lemma \ref{peak lemma}
(corresponding to the $\epsilon$ from above),
%. For $w \in \R$ let $K_{w}(z)$ be the reproducing kernel in $H(E)$ at the point $w$.
and set $f_{w, x}(z)=K_{w}(z)h(x,z)$.
Clearly,
$|f_{w, x}(y)|e^{-(1+\epsilon)\Phi(y)} \rightarrow 0$ as $\abs{y} \rightarrow \infty$,
and so   \eqref{the condition for func in F^0} holds for this function.
In particular, we have $h(x,x)=1$, so
  the formula gives pointwise
\begin{equation}\label{the series}
K_{w}(x)e^{-\Phi  (x)}=\sum_{\l \in \L}e^{\epsilon(\Phi(x)-\Phi
(\l))}\sqrt{\phi'(\l)}h(x,\l)g_{\l}(x)\widetilde{K}_{\l}(w), \qquad
\forall x\in \R.
\end{equation}
Moreover, by the
estimate \eqref{int estimate for h} of Lemma \ref{peak lemma}, the functions
\[
\zeta_{\l}(x)=e^{\epsilon(\Phi  (x)-\Phi
(\l))}\sqrt{\phi'(\l)}h(x,\l)g_{\l}(x)
\]
belong to $L^2(\R)$.

Now, to show that the series in (\ref{the series}) converges in
$L^2(\R)$, it suffices to verify that the sequence $\{\zeta_{\l}\}$ is a
Bessel sequence in $L^2(\R)$ and that the sequence
$\{\widetilde{K}_{\l}(w)\}_{\l\in\L}$, for fixed $w$, is in
$l^2(\L)$. The last claim follows from Corollary \ref {separation and sampling},
while the first holds as
  $\sum_{\l\in
\L}|g_{\lambda}(x)|\leq K$.

Given $g\in H(E)$, we have
$g(x)e^{-\Phi  (x)}\in L^2(\R)$. Therefore,   the equality
(\ref{the series}) yields
$g(w)= \langle g(x), K_w(x)\rangle_{H(E)}
=\sum
a_{\l}\widetilde{K}_{\l}(w)$. Hence, the Bessel property of $\{ \zeta_\lambda \}$ implies that
\[
\sum_{\l \in \Lambda}|a_{\l}|^2\lesssim\|g(x)e^{-\Phi
(x)}\|_{L^2(\R)}^2=\|g\|_{H(E)}^2.
\]
This concludes the proof.
\end{proof}

%=========================================================================
%=========================================================================
%=========================================================================
%=========================================================================
%=========================================================================
%=========================================================================

Following the approach of Beurling \cite[pp. 348--358]{beurling1989}, we treat sampling sequences
in $\mathcal{F}_{\Psi}^\infty$ by looking at certain translations. We need to consider
translates, both of the sequence on which we sample and the function being sampled. Since our
spaces are not translation invariant,   we modify the concept of translations accordingly.

As in Section \ref{sec_Basic_constructions}, we consider real valued subharmonic functions   $\Psi$, for which
 $\mu = \psi'(t) \dif t$ is a   doubling measure on $\R$ induced by $\Delta\Psi$  for some suitable function $\psi \in C^\infty(\R)$. Moreover, we suppose that $\psi$ satisfies  both   (b)  and (c) of Lemma \ref{local doubling equivalences}. I.e., there exist constants such that
 \begin{align} \label{sec 4 phi prime}
	\psi'(x) \simeq \psi'(y) \; \; \text{whenever} \; \; \abs{\psi(x) - \psi(y)} \leq 1.
\end{align}
and
 \begin{align} \label{sec 4 phi double prime}
	\abs{\psi''(x)} \lesssim (\psi'(x))^2.
\end{align}

Given a point $x_0\in \R$ we denote by $\tau_{x_0}:\C\rightarrow
\C$ the  {scaled translation}
% by $x_0$}, defined by
\begin{equation*}
\tau_{x_0}(z)=(\psi'(x_0))^{-1}z+x_0.
\end{equation*}
With this, given a set $\Lambda$, we make the following definitions,
\begin{align*}
	\L_{x_0}&=\{\tau^{-1}_{x_0}(\l):\l\in\L\}, \\
 \psi_{x_0}&=\psi\circ \tau_{x_0}-\psi({x_0}), \\
 	\mu_{x_0}(I) &=\mu( \tau_{x_0}(I)).
 \end{align*}
Note that evaluating  a function $f$ on the sequence $\L$ is the same as evaluating
$f\circ \tau_{x_0}$ on   $\L_{x_0}$.  Also, $\psi_{x_0}(0)=0$ and $\psi'_{x_0}(0)=1$, and the
Radon-Nikodym derivative of   $\mu_{x_0}$ is $\psi'_{x_0}$.

To translate  $\Psi$, we consider the operator
\[
K[f](z)=  \frac{1}{2\pi} \int_{\R} \left[\log
\Abs{1-\frac{z}{t} } -\textrm{Re}\big(P_q({z}/{t})\big)1\!\!1_{\R\setminus
(-1,1)}(t)\right]  f(t) \dif t,
\]
where $P_q$ is the Taylor polynomial of degree $q$ of $\log(1-x)$
around $x=0$. When it exists, the image $K[f]$ is a subharmonic function which
satisfies $\Delta K[f]= f(x)\dif x$.
To make the integral   converge for $f = \psi'$, observe that
since $\psi$ gives a doubling measure,  by Remark \ref{technicallemma},  there exists some $\gamma>0$ such that
$\psi'(x)\lesssim 1+ |x|^{\gamma}$ for all $x\in \R$ . A basic estimate ensures that the choice $q>\gamma+1$  is sufficient. (The number $\gamma$
depends only on the doubling constant.) 
%   so this   holds also for any $\psi_{x_0}$ and $\psi^*$.)
Given $x_0\in \R$, we now define
\begin{equation*}
	\Psi_{x_0}:=K[\psi'_{x_0}]. %-K[\psi'_{x_0}](0)
\end{equation*}

Recall, that the Fr\'{e}chet distance between two closed sets
$E,F\subseteq \R$ is defined by $[E,F]:=\max\{\sup_{x\in E}
d(x,F),\sup_{y\in F} d(y,E)\}$. We use Beurling's definition of
strong and compactwise limits of sets. A sequence $\{Q_j\}$ of
closed sets in $\R$, converges strongly to a set $Q\subseteq \R$,
denoted $Q_j\rightarrow Q$, if $[Q,Qj]\rightarrow 0$. The sequence
$Q_j$ converges compactwise to $Q$, denoted $Q_j \rightharpoonup
Q$, if for every compact set $K$ in $\R$ we have $(K\cap Q_j)
\cup\partial K\rightarrow (K\cap Q)\cup\partial K$.

 Given a set $\L \subseteq \R$, we say that $\L^*\subseteq \R$ is a
weak limit of $\L$, if there exists a sequence $\{x_k\}\subseteq
\R$ such that $\L_{x_k}\rightharpoonup\L^*$.

%==========================================================

\begin{lemma}\label{w-l of L}
Let $\psi$ be as above, and suppose that the sequence $\Lambda$ is  $\psi$-separated. If  $\{x_k\}\subseteq \R$ is a sequence, then there
exists a subsequence $\{x_{k_n}\}$ such that the following holds.
\begin{itemize}
\item[(i)]For some sequence $\Lambda^\ast$, we have  $\L_{x_{k_n}}\rightharpoonup\L^*$.
\item[(ii)]\label{w-l of eta} For  a function $\psi^*$, it holds that
$\psi_{x_{k_n}}\rightarrow \psi^*$ and $\psi'_{x_{k_n}}\rightarrow (\psi^*)'$ uniformly on compacts.
	\item[(iii)] Let $\psi^\ast$ be as above. Then it gives a doubling measure with doubling   constant smaller than the one of $\mu = \psi'(x)\dif x$. Moreover,    $(\psi^*)'$ satisfies \eqref{sec 4 phi prime}, with the same constants or better.
\item[(iv)]Let $\psi^\ast$ be as above and set $\Psi^\ast =K[(\psi^*)']$, then $\Psi_{x_{k_n}}\rightarrow \Psi^*$ uniformly on compacts in $\C$.
\end{itemize}
\end{lemma}

%==========================================================

\begin{remark}
We write
$(\L_{x_{k_n}},\Psi_{x_{k_n}},\psi_{x_{k_n}})\rightharpoonup
(\L^*,\Psi^*,\psi^*)$, and denote the set of all such week
limits by $W(\L,\Psi,\psi)$.
\end{remark}

%==========================================================

\begin{proof}
$(i)$: Fix some $a>0$. For $\tau_{x_k}^{-1}(\lambda) \in \Lambda_{x_k} \cap [-a,a]$, by Remark \ref{miamadre}, it holds that $d_\psi(x_k, \lambda) \lesssim a^\alpha$ for some $\alpha>0$. Since in addition $\Lambda$ is $\psi$-separated, we get
$\sup_k \abs{(\L_{x_{k}}\cap [-a,a])}<\infty$. Passing to a subsequence we can assume that the number of points in each such set
is constant. By a compactness argument we now find a set $\Lambda_a$ such that $\L_{x_{k_n}}\cap
[-a,a]\rightarrow \L_a$. Letting $a$ tend to infinity, we obtain $\Lambda^\ast$ from a usual diagonal argument.

$(ii)$:
For $t$ in an interval, there are constants, only depending on the euclidean length, such that  $\psi'(\tau_{x_k}(t))\simeq \psi'(x_k)$.
From this it follows that  the functions  $\psi'_{x_{k}}(t)$ and
 $\psi''_{x_{k}}(t)$ are  uniformly bounded on the interval. 
By the Arzel\`a-Ascoli theorem and Cantor diagonalization, there exists a subsequence $\{x_{k_n}\}$ and a function $\psi^*$ such
that $\psi'_{x_{k_n}}\rightarrow (\psi^*)'$ uniformly on compacts.

$(iii)$: By definition, the measure  $\mu_{x_k}$ is   doubling   with a
constant not bigger than the constant of $\mu$, and taking limits the statement about $\mu^\ast = (\psi^\ast)'(t) \dif t$ follows.
Similarly, to show that  \eqref{sec 4 phi prime} holds for $\psi^\ast$, it is sufficient to establish it for the $\psi_{x_k}$.

$(iv)$: To show that $\abs{K[\psi'_{x_k}]-K[(\psi^*)']}$ tends to zero on a given compact, split
the domain of integration into a large interval $(-r,r)$ and its complement.
The latter part is made small by choosing $r$ and the degree of the Taylor polynomial large
enough, depending on the doubling constant of $\psi'(x) \dif x$. The integral over $(-r,r)$ is then handled by
observing that the kernel of the   operator is integrable and letting $k$ tend to infinity.
\end{proof}

%==========================================================

In the next lemma, we list the properties preserved by weak limits.
 
 %==========================================================

\begin{lemma} \label{w-l of L2}
Let $\psi$ be as above, and suppose that the sequence $\Lambda$ is $\psi$-separated. If  $\{x_k\}\subseteq \R$ is a sequence such that   $(\psi_{x_k})' \rightarrow (\psi^\ast)'$ uniformly on compacts  for some function $\psi^\ast$, and $\Lambda_{x_k} \rightharpoonup \Lambda^\ast$  for some sequence $\Lambda^\ast$, then the following holds:
\begin{itemize}
	\item[(i)] $\Lambda^\ast$ is $\psi^\ast$-separated.
	\item[(ii)] $D^{-}_{\psi}(\L)\leq D^{-}_{\psi^*}(\L^*)$.
\end{itemize}
\end{lemma}

%==========================================================

\begin{proof}  
$(i)$: First, we note that, if $\L$ is $\psi$-separated by some constant,
 then each
$\L_{x_k}$ is $\psi_{x_k}$-separated by the same constant.
Using the convergence of $(\psi_{x_k})'$ and $\Lambda_{x_k}$ to $(\psi^\ast)'$ and $\Lambda^\ast$, respectively, the conclusion follows.

$(ii)$: By the definitions of $\mu_{x_k}$  and the density, it follows that for any $\epsilon> 0$ we can choose $r>0$ sufficiently large to get
\begin{equation*}
\frac{\min_{\mu_{x_k}
(I)=r} \abs{I\cap\L_{x_k}}}{r}
=\frac{\min_{\mu
(J)=r}\abs{J\cap\L}}{r} >D^{-}_{\psi}(\L)-\epsilon.
\end{equation*}
Observe that for big enough $k$, we have $|\abs{(I\cap
\L^*)}-\abs{(I\cap \L_{x_k})}|\leq 2$ (this is because Fr\'echet limits allows an extra element on each endpoint of $I$, and the sequences are uniformly separated).
In addition, if $\mu^*(I) = r$ then $\mu_{x_k}(I)>(1-\epsilon)r$ since
  $\psi_{x_k}\rightarrow\psi^*$ uniformly on compacts, and so we obtain for   big enough $r$ 
  that
 \[
\frac{\min_{\mu^*(I)=r}\abs{I\cap \L^*}+2}{r}\geq
(1-\epsilon)(D^{-}_{\psi}(\L)-\epsilon).
\]
By taking limits we get the result.
\end{proof}

%==========================================================

Given $x_0\in \R$, %the function  is harmonic in $\C$.
let $H_{x_0}$ be an entire
function with  $\Psi_{x_0}-\Psi\circ
\tau_{x_0}$ as its real part, % with the harmonic function this as a real part,
and  define an operator
$T_{x_0}:\mathcal{F}^{\infty}_{\Psi}\rightarrow
\mathcal{F}^{\infty}_{\Psi_{x_0}}$
by
\[
T_{x_0} :  f \longmapsto f\circ \tau_{x_0} e^{H_{x_0} }.
\]
The proof of the following lemma is   straight-forward and omitted.

%==========================================================

\begin{lemma} \label{w-l of f} Let $\Psi$ and $\psi$ be   as above.
\begin{itemize}
\item[(i)]$T_{x_0}$ is an isometry between $ \mathcal{F}^{\infty}_{\Psi}$
and $ \mathcal{F}^{\infty}_{\Psi_{x_0}}$. Moreover
\begin{equation}
\label{isometry} |T_{x_0}f(z)|e^{-\Psi_{x_0}(z)}=
|f \circ \tau_{x_0}(z)|e^{-\Psi \circ \tau_{x_0}(z)}.
\end{equation}
\item[(ii)] Suppose that $\{ f_k \}$
 is a sequence of functions in $\mathcal{F}^{\infty}_{\Psi}$, uniformly bounded
in norm, and let $\{x_k\}\subseteq \R$. If
$\psi'_{x_k}\rightarrow(\psi^*)'$ uniformly on compacts,  then there exists a subsequence
$\{x_{k_n}\}$ for which $T_{x_{k_n}}f_{k_n}$ converges uniformly
on compacts to a function in $ \mathcal{F}^{\infty}_{\Psi^*}$.
\end{itemize}
\end{lemma}

%=====================================================================
%=====================================================================
%=====================================================================
%=====================================================================
%=====================================================================
%=====================================================================

%\subsubsection{Applications for sampling}

It is not difficult to see that if $\L$ is a $\psi$-separated
sampling sequence for $\mathcal{F}_{\Psi}^{\infty}$ and
$(\L^*,\Psi^*,\psi^*)\in W(\L,\Psi,\psi)$, then $\L^*$ is
a sampling sequence for $\mathcal{F}_{\Psi^*}^{\infty}$ and, in
particular, a uniqueness set for the space. We will need the
opposite direction of this claim, which is given in the following
lemma.

%=====================================================================

\begin{lemma}\label{finaly using w-l thing}
Let $\Psi$ and $\psi$ be   as above, and suppose that the sequence $\L$ is $\psi$-separated.
%a $\psi$- separated sequence.
If the sequence
$\L^*$ is a uniqueness set for $\mathcal{F}_{\Psi^*}^{\infty}$ whenever
$(\L^*,\Psi^*,\psi^*)\in W(\L,\Psi,\psi)$, then
$\L$ is a sampling sequence for $\mathcal{F}_{\Psi}^{\infty}$.
\end{lemma}

%=====================================================================

\begin{proof}
Assume that $\L$ is not a sampling sequence for
$\mathcal{F}_{\Psi}^{\infty}$. This means that for every $k\in \N$ ,
there exists a function $f_k$  of
unit norm and   $x_k\in \R$ such that 
\[
\sup_{\l\in \L}|f_k(\l)|e^{-\Psi (\l)}\leq \frac{1}{k} \qquad \text{and} \qquad |f_k(x_k)|e^{-\Psi(x_k)}> 1/2.
\]
By lemma \ref{w-l of L},  there exists a subsequence
$\{x_{k_n}\}$ such that
$(\L_{x_{k_n}},\Psi_{x_{k_n}},\psi_{x_{k_n}})\rightharpoonup
(\L^*,\Psi^*,\psi^*)$. By Lemma \ref{w-l of f} (taking another
subsequence if needed), $T_{x_{k_n}}f_{k_n}$ converges uniformly
on compacts to a function $f^*\in \mathcal{F}^{\infty}_{\Psi^*}$. We
claim that on the one hand, $f^*\not \equiv 0$ and on the other hand,
$f^*|\L^*=0$. This contradicts our assumption and therefore
completes the proof.  We omit the details.
\end{proof}

%=====================================================================

We can now use the previous lemma to connect between the density of a
sequence and its sampling properties.

%=====================================================================

\begin{lemma}\label{finaly close to the end}
Let $\Psi$ and $\psi$ be as above, and suppose that the sequence $\L$ is $\psi$-separated.
If $D^{-}_{\psi}(\L)>1/2\pi$, then
$\L$ is a sampling sequence for $\mathcal{F}^{\infty}_{\Psi}$.
\end{lemma}

%=====================================================================

\begin{proof}
By Lemma \ref{finaly using w-l thing}, it is enough to show that if
$(\L^*,\Psi^*,\psi^*)\in W(\L,\Psi,\psi)$, and
 $f\in \mathcal{F}_{\Psi^*}^{\infty}$ satisfies $f|\L^*=0$, then $f\equiv 0$.
 Assume to the contrary
that there exists such $f$ which is non-zero. We may
assume, without loss of generality, that $f(0)\neq 0$ (otherwise
we translate   the setting as above). Choose $r>0$, by
Jensen's formula applied to $f$, we get
\[
\int_0^r\frac{\abs{\L^*\cap [-t,t]}}{t}dt\leq
 \frac{1}{2\pi}\int_0^{2\pi}\Psi^*(re^{i\theta})d\theta
+\textrm{O}(1).
\]
Since
$\Delta(\Psi^*(rz))=r(\psi^*)'(rx) \dif x$, Greens formula yields that the integral on the right-hand side of
the last expression is equal to
\begin{align*}
\frac{1}{2\pi}r\int_{-1}^{1}\log\frac{1}{\abs{x}}\cdot(\psi^*)'(rx)dx
&=
\frac{1}{2\pi}\int_{1\leq |t|\leq
r}\frac{\psi^*(t)}{\abs{t}}dt+\textrm{O}(1) \\
&=
\frac{1}{2\pi}\int_1^r\frac{\mu^*([-t,t])}{t}dt+\textrm{O}(1).
\end{align*}
This contradicts our assumption about the density of $\L$, which,
by Lemma \ref{w-l of L2}, implies that there exists some
$\epsilon>0$, so that for any $t$ big enough,
\[
\frac{\abs{\L^*\cap [-t,t]}}{\mu^*([-t,t])}\geq
\frac{1+\epsilon}{2\pi}.
\]
This ends the proof.
\end{proof}

%=====================================================================

We are now ready to prove the main result.

\begin{proof}[Proof of Theorem \ref{theorem_sampling}]
Let $\Phi$ be defined by \eqref{usual psi definition},  and recall that $\Delta\Phi = 2\phi'(x) \dif x$. Since
$D^{-}_{2\phi}(\L)>\frac{1}{2\pi}$, there exists an
$\epsilon>0$  small enough for
$D^{-}_{(1+\epsilon)2\phi}(\L)>\frac{1}{2\pi}$ to hold. Set
$\Psi=(1+\epsilon)\Phi$, and observe that  $\Delta \Psi =
 (1+\epsilon)2\phi'(x) \dif x$. So, by Lemma \ref{finaly
close to the end}, the sequence $\L$ is   sampling   for
$\mathcal{F}^{\infty}_{\Psi}$. The result now follows from Lemma \ref{tango}.
\end{proof}

%=====================================================================
%=====================================================================
%=====================================================================

\section*{Acknowledgement}

    The authors would like to thank Joaquim Ortega-Cerd\`a, Kristian Seip, Alexandru Aleman and Anton Baranov for some helpful conversations.

%=====================================================================
%=====================================================================
%=====================================================================

\bibliographystyle{amsplain}
\def\cprime{$'$} \def\cprime{$'$} \def\cprime{$'$}
\providecommand{\bysame}{\leavevmode\hbox to3em{\hrulefill}\thinspace}
\providecommand{\MR}{\relax\ifhmode\unskip\space\fi MR }
\providecommand{\MRhref}[2]{%
  \href{http://www.ams.org/mathscinet-getitem?mr=#1}{#2}
}
\providecommand{\href}[2]{#2}

\end{document}